\documentclass[11pt,a4paper]{amsart}
\usepackage{amsthm,amsmath,amssymb,amsfonts}
\usepackage{graphicx}
\usepackage[margin=3cm]{geometry}
\usepackage{bm}

\usepackage{dsfont} 

\numberwithin{equation}{section}

\newtheorem{theorem}{Theorem}[section]
\newtheorem{lemma}[theorem]{Lemma}
\newtheorem{proposition}[theorem]{Proposition}
\newtheorem{corollary}[theorem]{Corollary}
\newtheorem{example}[theorem]{Example}
\newtheorem{remark}[theorem]{Remark}

\newcommand{\E}{\mathbb{E}}
\renewcommand{\P}{\mathbb{P}}
\newcommand{\R}{\mathbb{R}}
\renewcommand{\S}{\mathbb{S}}
\newcommand{\1}{\mathds{1}}

\DeclareMathOperator{\argmax}{argmax}
\DeclareMathOperator{\Corr}{Corr}
\DeclareMathOperator{\Cov}{Cov}
\DeclareMathOperator{\detr}{detr}
\DeclareMathOperator{\diag}{diag}
\DeclareMathOperator{\sgn}{sgn}
\DeclareMathOperator{\spn}{span}
\DeclareMathOperator{\Sym}{Sym}
\DeclareMathOperator{\tr}{tr}
\DeclareMathOperator{\Unif}{Unif}
\DeclareMathOperator{\Var}{Var}
\DeclareMathOperator{\Vol}{Vol}

\newcommand{\cond}{{\,|\,}}
\newcommand{\dd}{\mathrm{d}}
\renewcommand{\vec}{\mathrm{vec}}

\newcommand{\bc}{b_\mathrm{cri}}
\newcommand{\tc}{\theta_\mathrm{cri}}
\newcommand{\tlc}{\theta_\mathrm{cri}'}
\renewcommand{\Mc}{M_\mathrm{cri}}
\newcommand{\Msupp}{M_\mathrm{supp}}
\newcommand{\Ptube}{\mathbb{P}_{\mathrm{tube}}}

\newcommand{\G}{\widetilde G}
\newcommand{\HH}{\widetilde H}
\newcommand{\RR}{\widetilde R}
\newcommand{\g}{\widetilde g}
\newcommand{\h}{\widetilde h}
\newcommand{\uu}{\widetilde u}
\newcommand{\vv}{\widetilde v}
\newcommand{\ww}{\widetilde w}
\newcommand{\X}{\check X}
\newcommand{\sigmaM}{{\sigma}M}
\newcommand{\Hf}{L}

\newcommand{\Xmax}{X_{\mathrm{max}}}
\newcommand{\Ymax}{Y_{\mathrm{max}}}
\newcommand{\hu}{\underline h}

\title[]{The volume-of-tube method for Gaussian random fields with inhomogeneous variance}
\author{Satoshi Kuriki}
\address{The Institute of Statistical Mathematics, 10-3 Midoricho, Tachikawa, Tokyo 190-8562, Japan}
\email{kuriki@ism.ac.jp}
\author{Akimichi Takemura} 
\address{The Center for Data Science Education and Research, 
Shiga University, 1-1-1 Banba, Hikone, Shiga 522-8522, Japan}
\email{a-takemura@biwako.shiga-u.ac.jp}
\author{Jonathan E.~Taylor}
\address{Department of Statistics, Sequoia Hall, 390 Jane Stanford Way, Stanford University, Stanford, CA 94305-4020, USA}
\email{jonathan.taylor@stanford.edu}

\begin{document}

\begin{abstract}
The tube method or the volume-of-tube method approximates the tail probability of the maximum of a smooth Gaussian random field with zero mean and unit variance.
This method evaluates the volume of a spherical tube about the index set,
and then transforms it to the tail probability.
In this study, we generalize the tube method to a case in which the variance is not constant.
We provide the volume formula for a spherical tube with a non-constant radius in terms of curvature tensors, and the tail probability formula of the maximum of a Gaussian random field with inhomogeneous variance, as well as its Laplace approximation.
In particular, the critical radius of the tube is generalized for evaluation of the asymptotic approximation error.
As an example, we discuss the approximation of the largest eigenvalue distribution of the Wishart matrix with a non-identity matrix parameter.
The Bonferroni method is the tube method when the index set is a finite set.
We provide the formula for the asymptotic approximation error for the Bonferroni method when the variance is not constant.
\end{abstract}

\maketitle

\section{Introduction: Tube with non-constant radius}

Let $M\subset\S^{n-1}$ be a $d$-dimensional $C^2$-closed submanifold of the ($n-1$)-dimensional unit sphere $\S^{n-1}=S(\R^n)=\{u\in\R^n \mid \Vert u\Vert=1\}$.
We assume that the linear-hull of $M$ is $\R^n$, and $M$ is a proper subset of $\S^{n-1}$ (that is, $d< n-1$).
Define a Gaussian random field on $M$ by
\begin{equation}
\label{Xu}
 X(u) = \sigma(u) \langle u,\xi\rangle, \quad u\in M\subset\S^{n-1},
\end{equation}
where $\xi\in\R^n$ is a random vector distributed as the standard Gaussian distribution $\mathcal{N}_n(0,I_n)$, $\langle \cdot, \cdot \rangle$ denotes the standard inner product of $\R^n$, and $\sigma(\cdot)$ is a positive $C^2$-function on $M$.
Then, $X(\cdot)$ is a mean-zero Gaussian random field on $M$ with variance $\Var(X(u))=\sigma(u)^2$ and the correlation function $\Corr(X(u),X(v))=\langle u,v\rangle$.
(\ref{Xu}) is the canonical representation of a Gaussian random field with a finite Karhunen-Lo\`eve expansion and a smooth sample path.

The purpose of this study is to provide a formula for approximating the tail probability of the maximum
\begin{equation}
\label{tailprob}
 \P\biggl( \max_{u\in M} X(u) > c \biggr) \quad \mbox{when $c$ is large}.
\end{equation}
When the variance is constant, say, $\sigma(u)\equiv 1$,
the recipes to approximate the tail probability (\ref{tailprob}) are well established and known as the volume-of-tube method, or simply, the tube method (\cite{knowles-siegmund:1989},\cite{johansen-johnstone:1990},\cite{sun:1993},\cite{kuriki-takemura:2001,kuriki-takemura:2009},\cite{takemura-kuriki:2002,takemura-kuriki:2003}).
The tail probability of (\ref{tailprob}) corresponds to the $p$-value of a max-type statistical test and appears in various statistical scenarios,
such as construction of simultaneous confidence bands, 
likelihood ratio test in singular models,
change point problems,
order-restricted statistical inference,
projection pursuit/testing Gaussianity, and so on
(e.g., \cite{naiman:1986,naiman:1990},\cite{sun:1990},\cite{kuriki-takemura:2008},\cite{kato-kuriki:2013},\cite{lu-kuriki:2017}).
The tube method provides sufficiently accurate $p$-value formulas for such problems.
 
In this study, we extend the tube formula to Gaussian random fields with inhomogeneous variance.
For later purposes, we define a standardized random field on $M$ by
\[
 Y(u)=\sigma(u)\bigl\langle u,\xi/\Vert\xi\Vert\bigr\rangle,\ \ u\in M,\ \ \xi\sim\mathcal{N}_n(0,I_n).
\]
We sometimes use the abbreviations
\[
 \Xmax = \max_{u\in M} X(u), \quad \Ymax = \max_{u\in M} Y(u).
\]
Note that the support of $\Ymax$ is $|\Ymax|\le \max_{u\in M}\sigma(u)$.
Because $\Vert\xi\Vert$ and $\xi/\Vert \xi\Vert$ are independently distributed,
we have
\begin{equation}
\label{PT}
 \P\bigl(\Xmax > c\bigr) = \E\biggl[\P\biggl(\Ymax > \frac{c}{\Vert \xi\Vert} \mid \Vert \xi\Vert\biggr)\biggr].
\end{equation}
Moreover, because $\xi/\Vert\xi\Vert\sim\Unif(\S^{n-1})$, the uniform distribution on $\S^{n-1}$,
\begin{align}
\label{PU}
 \P\bigl(\Ymax > b\bigr)
&= \P\biggl(\max_{u\in M} \sigma(u) \langle u,\xi/\Vert \xi\Vert\rangle > b\biggr) \\
&= \P\Bigl(\exists u\in M,\,\langle u,\xi/\Vert \xi\Vert\rangle > b/\sigma(u)\Bigr) \nonumber \\
&= \frac{1}{\Omega_n}\Vol_{n-1}\bigl((M)_{b/\sigma(\cdot)}\bigr), \nonumber
\end{align}
where
\begin{align}
\label{tube}
 (M)_{f}
=& \bigl\{ w\in \S^{n-1} \mid \exists u\in M,\,\langle u,w\rangle > f(u) \bigr\},
\end{align}
$\Vol_{n-1}(\cdot)$ is the ($n-1$)-dimensional volume, and
\begin{equation}
\label{Omega} 
 \Omega_n = \Vol_{n-1}(\S^{n-1}) = \frac{2\pi^{n/2}}{\Gamma(n/2)}
\end{equation}
is the volume of ($n-1$)-dimensional unit sphere.
Further, we observe that the distribution of $\Xmax$ is derived from that of $\Ymax$ in (\ref{PU}), which is proportional to the volume of the spherical set $(M)_{b/\sigma(\cdot)}$ in (\ref{tube}).
(It is shown that the relation between $\P\bigl(\Ymax > \cdot\bigr)$ and $\P\bigl(\Xmax > \cdot\bigr)$, i.e., (\ref{PT}), is essentially the Laplace/inverse Laplace transforms, and hence one-to-one.)

\begin{figure}[ht]
\begin{center}
\begin{tabular}{ccccc}
\scalebox{0.45}{\includegraphics{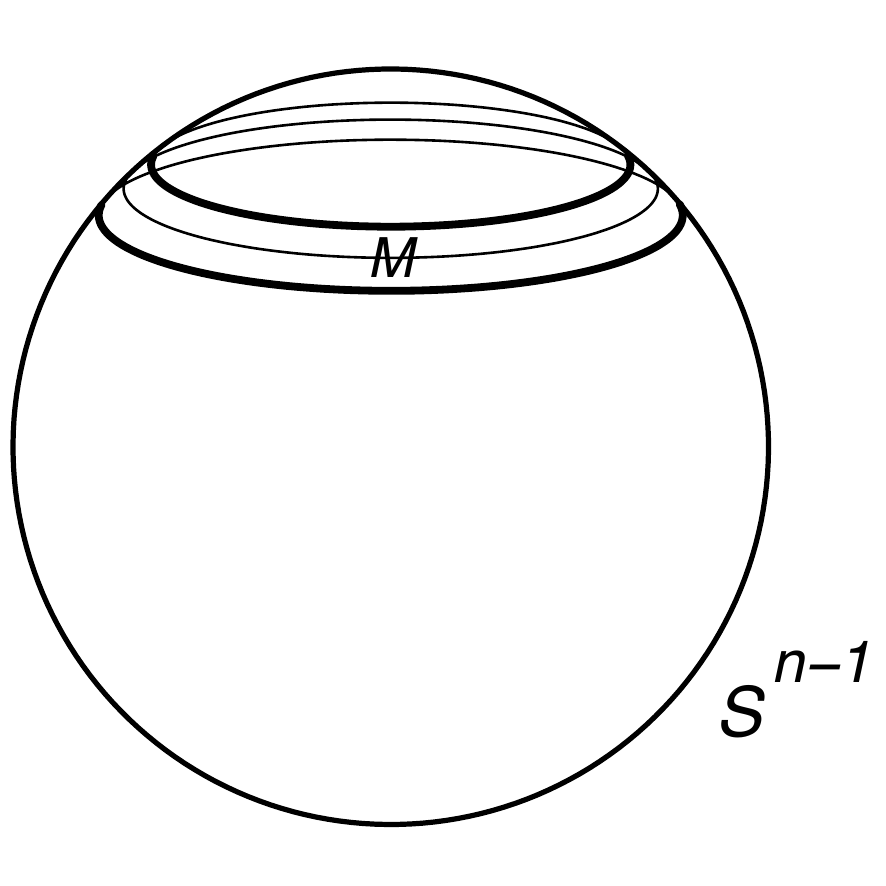}}
&&
\scalebox{0.45}{\includegraphics{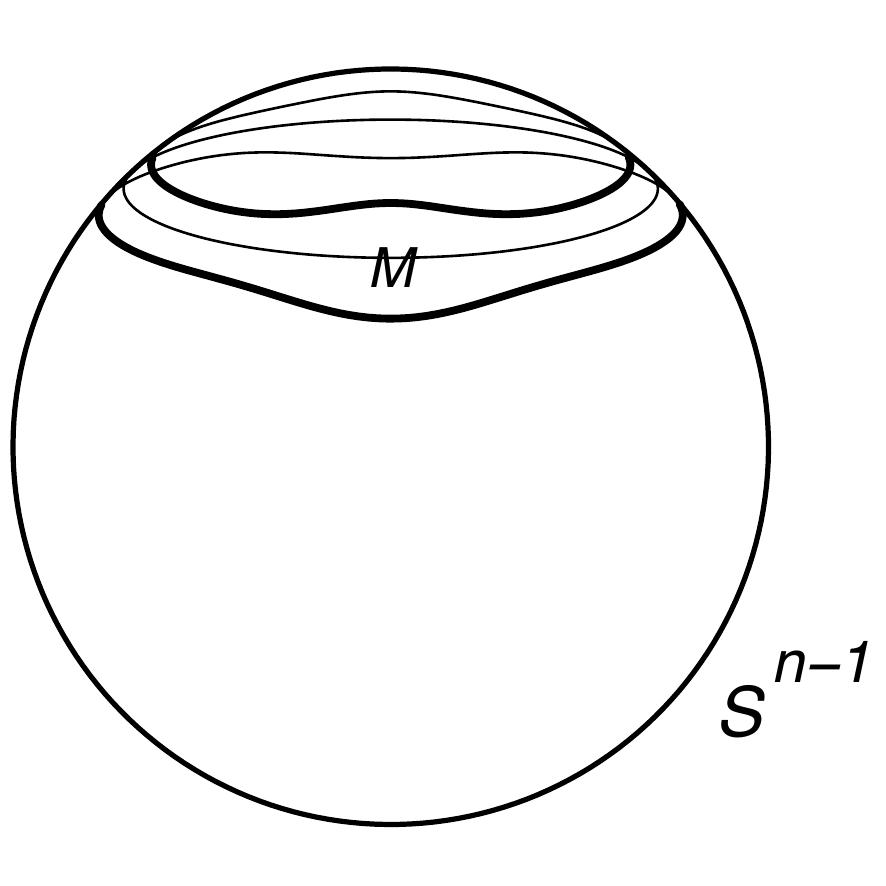}}
&&
\scalebox{0.45}{\includegraphics{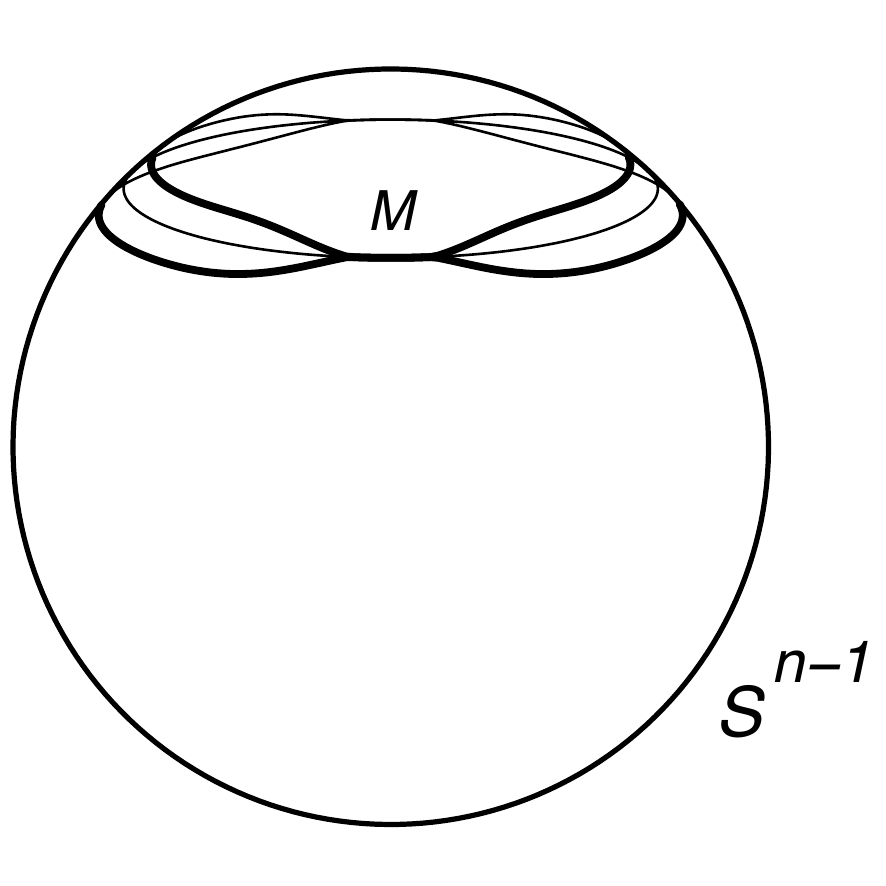}}
\end{tabular}
\caption{Tubes $(M)_{f}$.\newline
(left: homogeneous variance, center and right: inhomogeneous variance)
}
\label{fig:tubes}
\end{center}
\end{figure}

Figure \ref{fig:tubes} depicts examples of $(M)_{f}$ with $f(\cdot)=b/\sigma(\cdot)$.
$(M)_{f}$ is a tube-like area about the set $M$.
When $\sigma(u)$ is constant, the width of $(M)_{f}$ is constant,
and its volume can be evaluated using the tube-volume formula by
H.~Hotelling \cite{hotelling:1939} and H.~Weyl \cite{weyl:1939}.
The original tube method for constant variance is used to evaluate the volume of (\ref{tube}) via the tube-volume formula and thereafter transforms it into the distribution of $\Xmax$ by (\ref{PU}) (see \cite{sun:1993},\cite{kuriki-takemura:2001}).
As in the right panel of Figure \ref{fig:tubes},
the radius of tube may vanish.
We call this phenomenon the degeneracy of tube.
We will see later the condition that such degeneration occurs.
In this study, we will extend Weyl's tube-volume formula and derive the distribution of $\Xmax$ in a similar manner.

The tube treated here is spherical on the unit sphere.
The volume of a tube with non-constant radius when the ambient space is $\R^n$ is discussed under different motivations
 (e.g., \cite{roccaforte:2013},\cite{rosmarin:2015}).

To approximate the tail probabilities of the maxima of random fields,
R.~Adler and K.~Worsley founded a methodology based on the Euler-Poincar\'e characteristic of the excursion set (\cite{adler-hasofer:1976,adler:1981},\cite{worsley:1995}).
Their method is referred to as the expected Euler characteristic heuristic or the Euler characteristic method.
Later, this method has been developed and generalized in theoretical and application aspects (e.g., \cite{taylor-adler:2003,taylor-adler:2009}, \cite{taylor:2006}, \cite{taylor-worsley:2006,taylor-worsley:2013}).
For comprehensive treatments of this method, we refer to \cite{adler-taylor:2007} and \cite{adler-taylor:2011}.
\cite{takemura-kuriki:2002} showed that the tube method and the Euler characteristic method provide the same formula for the tail probability in (\ref{tailprob}) with $\sigma(u)\equiv 1$.

As a related problem, \cite{taylor-etal:2016} formulated a test statistic for a global null hypothesis in a regularized regression
as the maximum of a Gaussian random field with inhomogeneous variance:

\begin{equation}
\label{TLT}
 \max_{\eta\in K}\langle\eta,\xi\rangle, \quad \xi\sim \mathcal{N}_n(0,I_n),
\end{equation}
where $K$ is a closed, convex, stratified submanifold of $\R^n$.
In our notation, (\ref{TLT}) is
$\max_{u\in M}X(u)$ with $X(u)=\sigma(u)\langle u,\xi\rangle$, $M=\{\eta/\Vert\eta\Vert\mid \eta\in K\}$ and $\sigma(u)=\Vert\eta\Vert$ for $u=\eta/\Vert\eta\Vert$.
They derived the conditional distribution of the maximum when  $\eta^*=\argmax_{\eta\in K}\langle\eta,\xi\rangle$ is given,
and constructed a so-called Kac-Rice test based on the conditional distribution.
Differently from \cite{taylor-etal:2016}, we treat the unconditional distribution of the maximum.
Another difference is the assumption of convexity.
In this paper, $M$ is a closed submanifold and not (geodesically) convex.
We will examine the effects of the non-convexity of the set
\begin{equation}
\label{sigmaM}
 \sigmaM = \bigl\{ \sigma(u)u \in\R^n \mid u\in M \bigr\}.
\end{equation}
It will be shown that only the supporting points of $\sigmaM$ make contributions to the volume of tube formula, and the degeneracy of the tube (the right panel of Figure \ref{fig:tubes}) occurs at the non-supporting points of $\sigmaM$.

The outline of the paper is as follows.
The main results are summarized in Section \ref{sec:main}.
We present the approximate formulas for the tail probability of $\Ymax$ and $\Xmax$.
To evaluate their approximation error, 
we generalize the notion of the critical radius of a tube.
Consequently, based on the resulting formula, we provide its Laplace approximation under the assumption that the set
\begin{equation}
\label{M0}
 M_0 = \Bigl\{ u\in M \mid \sigma(u) = \max_{u\in M}\sigma(u) \Bigr\}
\end{equation}
forms a closed $C^2$-manifold.
In addition, we discuss the Bonferroni method, which is the tube method when the index set is finite.
We provide the formula for the asymptotic approximation error for the Bonferroni method when the variance is not constant.
Examples are provided in Section \ref{sec:examples}.
One of them is the tail probability of the largest eigenvalue of the Wishart matrix when the parameter matrix is not an identity matrix, which will be of independent interest.
All proofs for the main results are provided in Section \ref{sec:proofs}.

\section{Main Results}
\label{sec:main}

\subsection{Preliminary and notations}

The index set $M$ is treated as a $C^2$-submanifold endowed with the metric induced by the ambient space $\R^n$.
Suppose that each point $u\in M$ is represented by the local coordinates $t=(t^i)_{1\le i\le d}$ as $u=\varphi(t)$.
We use conventions $\partial/\partial t^i=\partial_i$,
$\varphi_i=\partial_i\varphi$,
$\varphi_{ij}=\partial_i\partial_j\varphi$, \textit{etc}.
The tangent space of $M$ at $u\in M$ is
\[
 T_u(M) = \spn\{\varphi_i \mid 1\le i\le d \}.
\]
The metric tensor is
$g_{ij} = \langle\varphi_i,\varphi_j\rangle$,
and the affine connection is
$\Gamma_{ij,k} = \langle\varphi_{ij},\varphi_k\rangle= (1/2)(\partial_i g_{jk} + \partial_j g_{ik} - \partial_k g_{ij})$.
The volume element of $M$ at $u\in M$ is $\det(G)^{\frac{1}{2}}\prod_i \dd t^i|_u$, where $G=(g_{ij})_{1\le i,j\le d}$ is the metric tensor in $d\times d$ matrix form.
When needed, we write $G=G_u$ to indicate the point $u\in M$ where the metric is defined.
The covariant derivative in the direction of $\partial/\partial t^i$ is denoted by $\nabla_i$.
The inverse of the metric matrix $G$ is denoted by $G^{-1}=(g^{ij})_{1\le i,j\le d}$.
Let
\begin{equation}
\label{R}
 R_{ij;kl} = \partial_i\Gamma_{jk,l} - \partial_j\Gamma_{ik,l}
 + \sum_{\alpha,\beta=1}^d (\Gamma_{ik,\alpha}\Gamma_{jl,\beta} - \Gamma_{il,\alpha}\Gamma_{jk,\beta}) g^{\alpha\beta}
\end{equation}
be the curvature tensor.

Define the normal space of $M$ at $u$ by
\begin{equation}
\label{Nu}
 N_u = T_u(M)^\perp\cap T_u(\S^{n-1}).
\end{equation}
Note that
\[
 \R^n = T_u(\R^n) = \spn\{u\}\oplus T_u(M) \oplus N_u,
\]
where $\oplus$ is the orthogonal direct sum.

Let
\begin{equation}
\label{ell}
 \ell = \ell(u) = \log\sigma(u),
\end{equation}
and let
\begin{align}
\label{C}
& C = C_u = (c_{ij})_{1\le i,j\le d}, \\ 
& c_{ij} = - \nabla_i\nabla_j\ell + \nabla_i\ell \nabla_j\ell
= -\biggl(\partial_i\partial_j\ell -\sum_{k,l=1}^d \Gamma_{ij,k}\partial_l\ell g^{kl}\biggr) + (\partial_i\ell)(\partial_j\ell). \nonumber
\end{align}

The gradient of $\ell(u)$ in (\ref{ell}) in the tangent space $T_u(M)$ is
\[
 \nabla\ell = \nabla\ell(u) = \nabla|_{T_u(M)}\ell(u) = \sum_{i,j=1}^d \nabla_i\ell\varphi_j g^{ij}.
\]
The squared length of the gradient is
$\Vert\nabla\ell\Vert^2 = \Vert\nabla\ell(u)\Vert^2 = \sum_{i,j=1}^d (\nabla_i\ell)(\nabla_j\ell) g^{ij}$.

The $k$th elementary symmetric function of the eigenvalues of a matrix is denoted by $\tr_k(\cdot)$ (\cite{muirhead:1982}) or $\detr_k(\cdot)$ (\cite{worsley:1995}).
Note that $\tr_1(\cdot)=\tr(\cdot)$, and $\tr_d(A)=\det(A)$ if $A$ is $d\times d$, and let $\tr_0(\cdot)=1$.
We use the relation
\[
 \det(I_d+A)=\sum_{k=0}^d \tr_k(A).
\]

For a finite set $I=\{ i_1,\ldots,i_n \}$ ($i_1<\cdots<i_n$), let
\begin{equation}
\label{SI}
 \Sigma(I) = \{ (\pi_1,\ldots,\pi_n) \mid \{\pi_1,\ldots,\pi_n\}=I \}
\end{equation}
be the set of all permutations of elements of $I$.
For $I$ such that $n=|I|$ is even, let
\begin{equation}
\label{PI}
 \Pi(I) = \bigl\{ (\pi_1,\ldots,\pi_n)\in \Sigma(I) \mid \pi_{1}<\pi_{2},\,\pi_{3}<\pi_{4},\,\ldots,\pi_{n-1}<\pi_{n} \bigr\}
\end{equation}
be the set of ordered parings of $I$.
Let $\Sigma_n=\Sigma(\{1,\ldots,n\})$ and $\Pi_n=\Pi(\{1,\ldots,n\})$.

\subsection{Volume of tube formula}

To state main theorems, we introduce the notion of self-overlap of tube.
Figure \ref{fig:tubes} depicts tube $(M)_{b/\sigma(\cdot)}$.
Suppose we let $b$ be small.
Subsequently, the width of the tube becomes large, and therefore causing self-overlap.
More formally, $(M)_{b/\sigma(\cdot)}$ is said to have a self-overlap if there exists a point $w\in (M)_{b/\sigma(\cdot)}$ such that the function $u\mapsto\sigma(u)\langle u,w\rangle$ has two local maxima at the points $u^*,u^{**}\in M$ and
$\sigma(u^*)\langle u^*,w\rangle=\sigma(u^{**})\langle u^{**},w\rangle$.
The threshold $b=\bc$ causing such overlap is defined by
\begin{equation}
\label{critical}
 \bc = \sup\bigl\{ b\ge 0 \mid (M)_{b/\sigma(\cdot)} \mbox{ has a self-overlap} \bigr\}.
\end{equation}

For $x\ge 0$, let
\[
 \bar B_{a,b}(x) =
\begin{cases}
 \displaystyle
 \frac{\Gamma(a+b)}{\Gamma(a)\,\Gamma(b)}\int_x^1 t^{a-1}(1-t)^{b-1}\,\dd t
 & (0\le x\le 1), \\
 0 & (x>1),
\end{cases}
\]
be the upper probability of the beta distribution with parameter $(a,b)$.
Now, we state the following main results.
\begin{theorem}[Volume of tube] 
\label{thm:volume}
Let $\dd u$ be the volume element of $M$ at $u\in M$.
Let $\bc$ be the threshold defined in (\ref{critical}).
For $b \ge \bc$,
\begin{align}
\label{PU1}
\P\bigl(\Ymax > b\bigr)
=& \frac{1}{\Omega_n}\Vol_{n-1}\bigl((M)_{b/\sigma(\cdot)}\bigr) \\
=& \sum_{e=0}^d
 \frac{1}{(2\pi)^{e/2} \Omega_{d-e+1}} \int_M \frac{\det(I_d+C_u G_u^{-1})}
 {(1+\Vert\nabla\ell(u)\Vert^2)^{\frac{1}{2}(d-e+1)}} \nonumber \\
& \qquad\qquad
 \times \bar B_{\frac{1}{2}(d-e+1),\frac{1}{2}(n-d+e-1)}
   \biggl(\frac{1+\Vert\nabla\ell(u)\Vert^2}{\sigma(u)^2} b^2\biggr)\,\zeta_e(u)\,\dd u,
\nonumber
\end{align}
where $\Omega_k$ is defined in (\ref{Omega}),
$G_u$ is the metric matrix at $u$, $C_u$ is defined in (\ref{C}),
\begin{equation}
\label{zeta}
 \zeta_e(u) =
\begin{cases}
\displaystyle
 \sum_{I\subset\{1,\ldots,d\},\,|I|=e} \frac{1}{(e/2)!} \sum_{\pi,\tau\in \Pi(I)}
 \sgn(\pi,\tau)\,
 \RR_{\pi_1 \pi_2}^{\tau_1 \tau_2}\cdots \RR_{\pi_{e-1} \pi_e}^{\tau_{e-1} \tau_e}\,
  & (\mbox{if $e$ is even}), \\
0 & (\mbox{if $e$ is odd}),
\end{cases}
\end{equation}
with
\begin{equation}
\label{tildeR}
 \RR_{ij}^{kl} =
 \RR_{ij}^{kl}(u) =
 \sum_{\alpha,\beta=1}^d \Bigl(R_{ij;\alpha\beta} - \bigl(g_{i\alpha}g_{j\beta}-g_{i\beta}g_{j\alpha}\bigr)\Bigr)\g^{\alpha k}\g^{\beta l},
\end{equation}
$\g^{ij}$ the $(i,j)$th element of the inverse matrix $\G_u^{-1}$ of $\G_u=G_u+C_u$,
\[
 \sgn(\pi,\tau) = \sgn
 \begin{pmatrix}\pi_1 & \cdots & \pi_e \\ \tau_1 & \cdots & \tau_e \end{pmatrix}
\]
the sign of permutation.
In particular, $\zeta_0(u)=1$ and
$\zeta_2(u) = \frac{1}{2}\sum_{i,j=1}^d \RR_{ij}^{ij}(u)$.

Moreover, (\ref{PU1}) holds even if the range $M$ of the integral is replaced with
\begin{equation}
\label{Mcri} 
 \Mc = \biggl\{ u\in M \mid 
   \frac{1+\Vert\nabla\ell(u)\Vert^2}{\sigma(u)^2} \bc^2 < 1 \biggr\}.
\end{equation}
\end{theorem}

\begin{remark}[Weyl's tube formula]
When $\sigma(u)\equiv 1$, (\ref{PU1}) becomes
\begin{align*}
\P\bigl(\Ymax > b\bigr)
=& \frac{1}{\Omega_n}\Vol_{n-1}\bigl((M)_b\bigr) \\
=& \sum_{e=0}^d
 \frac{1}{(2\pi)^{e/2} \Omega_{d-e+1}} \int_M  
\,\zeta_e(u)\,\dd u \times \bar B_{\frac{1}{2}(d-e+1),\frac{1}{2}(n-d+e-1)}
   \bigl(b^2\bigr),
\end{align*}
where $\zeta_e(u)$ is defined by (\ref{zeta}) with
\[
 \RR_{ij}^{kl} =
 \sum_{\alpha,\beta=1}^d R_{ij;\alpha\beta} g^{\alpha k}g^{\beta l}- \bigl(\delta_i^k\delta_j^l-\delta_i^l\delta_j^k\bigr),
 \qquad \delta_i^j = \begin{cases} 1 & (i=j), \\ 0 & (i\ne j). \end{cases}
\]
This is original Weyl's tube formula \cite{weyl:1939}.
\end{remark}

The theorem below characterizes the threshold $\bc$.
This is also useful for numerical calculation (Section \ref{sec:examples}).

\begin{theorem}
\label{thm:critical}
The threshold $\bc$ appearing in Theorem \ref{thm:volume} is
\begin{equation}
\label{critical1}
 \bc^2 = \sup_{u\in M} \bc'(u)^2, \qquad \bc'(u)^2 = \sup_{w\in M\setminus\{u\}} \frac{\sigma(u)^2}{1+h(u,w)},
\end{equation}
where
\begin{equation}
\label{h}
h(u,w) =
 \Vert\nabla\ell(u)\Vert^2 + \biggl( \frac{\{ \sigma(u)/\sigma(w)
   - \langle w,u-\nabla\ell(u)\rangle\}_+}{\Vert P_u^\perp w \Vert}\biggr)^2.
\end{equation}
Here $x_+ = \max\{x,0\}$, and
$P_u^\perp$ is the orthogonal projection in $\R^n$ onto the subspace $N_u$ defined in (\ref{Nu}).
That is,
\[
 \Vert P_u^\perp w \Vert^2 = 1 - \langle u,w\rangle^2 -\sum_{i,j=1}^d \langle\varphi_i,w\rangle\langle\varphi_j,w\rangle g^{ij}.
\]
\end{theorem}

\begin{remark}[Critical radius]
\label{rem:critical}
When $\sigma(u)\equiv 1$, (\ref{critical1}) becomes
\[
 \bc = \cos\tc = \sup_{u\in M}\cos\tlc(u), \quad \mbox{where}\ \ %
  \tan\tlc(u) = \inf_{w\in M\setminus\{u\}} \frac{1-\langle w,u\rangle}{\Vert P_u^\perp w\Vert}.
\]
$\tc$ and $\tlc(u)$ are referred to as the global and local critical radius of the tube about $M$, respectively.
They are positive when $M$ is a $C^2$-closed manifold (e.g., \cite[Remark 3.1]{kuriki-takemura:2001}).
We generalize the local critical radius $\tlc(u)$ to the inhomogeneous variance case by
\[
 \tan^2\tlc(u) = \inf_{w\in M\setminus\{u\}}h(u,w).
\]
Then,
\[
 \bc^2 = \sup_{u\in M}\bigl[\sigma(u)^2\cos^2\tlc(u)\bigr].
\]
\end{remark}

In the formula (\ref{PU1}), if $b$ is larger than $\sigma_0=\max_{u\in M}\sigma(u)$, then the formula holds trivially as $0=0$.
Therefore, Theorem \ref{thm:volume} is meaningless unless $\bc<\sigma_0$ holds as follows.

\begin{proposition}
\label{prop:validity}
The threshold $\bc$ in Theorem \ref{thm:volume} satisfies
\[
 \bc < \sigma_0 = \max_{u\in M}\sigma(u).
\]
\end{proposition}

From Proposition \ref{prop:validity}, we immediately have $M_0 \subset \Mc$,
where $M_0$ and  $\Mc$ are defined in (\ref{M0}) and (\ref{Mcri}), respectively.

\subsection{Supporting points of $\sigmaM$}

In this subsection, we focus on the shape of $\sigmaM$ in (\ref{sigmaM}).
Define a half space and its boundary in $\R^n$ by
\begin{align*}
 \Hf_u &= \{ z\in\R^n \mid \langle z,u-\nabla\ell(u)\rangle \le \sigma(u)\}, \\
 \partial \Hf_u &= \{ z\in\R^n \mid \langle z,u-\nabla\ell(u)\rangle = \sigma(u)\}.
\end{align*}
It is easy to see that
\begin{align*}
 \langle\sigma(u)u,u-\nabla\ell(u)\rangle = \sigma(u), \quad
 \biggl\langle\frac{\partial}{\partial t^i}\bigl[\sigma(u)u\bigr],u-\nabla\ell(u)\biggr\rangle = 0,
\end{align*}
that is,
the hyperplane $\partial\Hf_u$ is tangent to $\sigmaM$ at $\sigma(u)u$.
Let
\begin{align*}
 \Msupp = \bigl\{ u\in M \mid \partial\Hf_u \mbox{ is a supporting hyperplane of } \sigmaM \mbox{ with }\Hf_u\supset\sigmaM \bigr\}
\end{align*}
be the set of supporting points.

\begin{proposition}
\label{prop:supporting}
{\rm (i)} $u\in \Msupp \iff \sigma(u)/\sigma(w) - \langle w,u-\nabla\ell(u)\rangle\ge 0,\,\forall w\in M$.

{\rm (ii)} $\Msupp\supset\Mc$. 

{\rm (iii)} When $u\notin\Msupp$, the tube $(M)_{b/\sigma(\cdot)}$, $b\ge\bc'(u)$, is degenerated at $u$.
\end{proposition}

(i) means that, when $u\in\Msupp$, the argument of the truncation function $\{\cdot\}_+$ in (\ref{h}) is always nonnegative.
(ii) implies that the non-supporting points do not make contribution to the tube-volume formula (\ref{PU1}).
The inclusion in (ii) is strict in general.

\subsection{Tail probability of the maximum}

Let
\[
 \bar G_\nu(x) = \frac{1}{\Gamma(\nu/2)}\int_x^{\infty} t^{\nu/2-1}e^{-t/2}\,\dd t 
\]
be the upper probability of the chi-square distribution $\chi^2_\nu$ with $\nu$ degrees of freedom.
Let
\[
 \bar\Phi(x) = \frac{1}{(2\pi)^{1/2}}\int_x^\infty e^{-t^2/2}\,\dd t
\]
be the upper probability of the standard normal distribution $\mathcal{N}(0,1)$.

\begin{theorem}[Tube method]
\label{thm:tube_method}
Let $G$, $C$, $\dd u$, and $\bc$ be defined as in Theorem \ref{thm:volume}.
Then,
\begin{equation}
\label{PT1}
 \P\bigl(\Xmax > c\bigr) = \Ptube\bigl(\Xmax > c\bigr) + O(\bar G_n(c^2/\bc^2)),
\quad c\to\infty,
\end{equation}
where
\begin{align}
\label{PT2}
\Ptube\bigl(\Xmax > c\bigr)
= \sum_{e=0}^d
 \frac{1}{(2\pi)^{e/2} \Omega_{d-e+1}}
 \int_M \frac{\det(I_d+C_u G_u^{-1})}{(1+\Vert\nabla\ell(u)\Vert^2)^{\frac{1}{2}(d-e+1)}} \\
 \times \bar G_{d-e+1}\biggl( \frac{1+\Vert\nabla\ell(u)\Vert^2}{\sigma(u)^2} c^2 \biggr)\,\zeta_e(u)\,\dd u.
\nonumber
\end{align}
The range $M$ of the integral in (\ref{PT2}) can be replaced with $\Mc$ in (\ref{Mcri}).
\end{theorem}

Each term of the tube formula (\ref{PT2}) is of the order of
$\bar G_{\nu}(c^2/\sigma_0^2)\asymp c^{\nu-2}e^{-c^2/2\sigma_0^2}$,
$\sigma_0=\max_{u\in M}\sigma(u)$,
whereas the order of the remainder term is
$\bar G_n(c^2/\bc^2)\asymp c^{n-2}e^{-c^2/2\bc^2}$.
Here $f(c)\asymp g(c)$ denotes that both $|f(c)/g(c)|$ and $|f(c)/g(c)|$ are bounded when $c\to\infty$.
We use the notation $f(c)\sim g(c)$ if $f(c)/g(c)\to 1$ as $c\to\infty$.

Proposition \ref{prop:validity} affirms that the remainder term of the approximation (\ref{PT1}) is exponentially smaller than that of the main part.
That is,
\begin{equation}
\label{validity}
 \lim_{c\to\infty}\frac{1}{c^2}\log
 \biggl|\frac{\Ptube\bigl(\Xmax > c\bigr)}{\P\bigl(\Xmax > c\bigr)}-1\biggr|
 = -\frac{1}{2}\biggl(\frac{1}{\bc^2}-\frac{1}{\sigma_0^2}\biggr)<0.
\end{equation}

\begin{remark}
\cite[Theorem 4.3]{taylor-takemura-adler:2005} proved that (\ref{validity}) holds when $\sigma(u)\equiv 1$ and $X(u)$ does not necessarily have a finite Karhunen-Lo\`eve expansion, i.e., $n$ may be infinite.
In other words, the remainder term $O(\bar G_n(c^2/\bc^2))$ in (\ref{PT1}) does not depend on $n$.
For this purpose, they described the threshold $\bc$ in terms of $X(\cdot)$ irrespective of the dimension $n$ of the ambient space.

In the inhomegeneous variance case, by introducing a unit-variance Gaussian random field
\[
 \X(u) = \frac{1}{\sigma(u)}X(u), \quad u\in M,
\]
$h(u,w)$ in (\ref{h}) and hence $\bc$ in (\ref{critical1}) can be written in terms of $X(\cdot)$ by substituting
\begin{align*}
\langle w,u\rangle =& \Cov\bigl(\X(w),\X(u)\bigr),
\\
\langle w,\nabla\ell(u)\rangle =& \sum_{i,j=1}^d \Cov\bigl(\X(w),\partial_i \X(u)\bigr)\partial_j \ell(u) g^{ij}(u), \quad \partial_i=\partial/\partial t_i,
\\
\Vert P_u^\perp w \Vert^2 =& \Var\Bigl(\X(w)\cond \X(u),\bigl(\partial_i\X(u)\bigr)_{1\le i\le d}\Bigr)
\end{align*}
into (\ref{h}), where $(t^i)_{1\le i\le d}$ is the local coordinates of $M$ around $u$.
We conjecture that (\ref{validity}) holds even when $X(u)$ does not necessarily have a finite Karhunen-Lo\`eve expansion.
\end{remark}

In (\ref{PT2}), the term of $e=0$ is the leading term when $c\to\infty$.
In the case where $\sigma(u)$ is not constant, this leading term can
again be approximated using Laplace's method (\cite{erdelyi:1956}).

\begin{theorem}[Laplace approximation]
\label{thm:laplace}
Assume that $\sigma(u)$ takes its maximum $\sigma_0$ on the $d_0$-dimensional closed manifold in (\ref{M0}).
Let $\dd u_0(u)$ be the volume element of $M_0$ at $u$.
Let $G_u$ and $C_u$ be defined in Theorem \ref{thm:volume}.
Then,
\begin{align}
\label{laplace}
 \P\bigl(\Xmax > c\bigr) \sim
 \frac{1}{\Omega_{d_0+1}}
 \int_{M_0} \frac{\det(I_d+C_u G_u^{-1})^{\frac{1}{2}}}{\tr_{d-d_0}(C_u G_u^{-1})^{\frac{1}{2}}} \dd u_0(u)
 \times \bar G_{d_0+1}(c^2/\sigma_0^2), \quad c\to\infty.
\end{align}
\end{theorem}

\begin{corollary}
\label{cor:laplace}
Assume that $\sigma(u)$ takes the unique maximum $\sigma_0$ at $u_0$.
Let $G_0=G_{u_0}$, $C_0=C_{u_0}$.
Then,
\[
\P\bigl(\Xmax > c\bigr) \sim 
 \frac{\det(I_d + C_0 G_0^{-1})^{\frac{1}{2}}}{\det(C_0 G_0^{-1})^{\frac{1}{2}}}
 \,\bar\Phi(c/\sigma_0),
\quad c\to\infty.
\]
\end{corollary}
Corollary \ref{cor:laplace} is a special case of \cite[Theorem 8.2]{piterbarg:1996}.

\subsection{Bonferroni method as the tube method}

So far we assumed that the index set $M$ is a manifold.
In this subsection, we consider the case where $M$ is a finite set.

Let $(X_1,\ldots,X_K)$ be a Gaussian random vector with moment structure
$\E[X_i]=0$, $\Var(X_i)=\sigma_i^2$, $\Corr(X_i,X_j)=\rho_{ij}$, $|\rho_{ij}|<1$.
We consider the distribution of the maximum
\[
 \Xmax = \max\{X_1,\ldots,X_K\}.
\]
Such $X_i$'s are realized by
\[
 X_i = \sigma_i \langle \varphi^{(i)},\xi\rangle, \quad \xi\sim\mathcal{N}_n(0,I_n),
\]
where $\varphi^{(i)}$'s are unit vectors such that
$\rho_{ij}=\langle\varphi^{(i)},\varphi^{(j)}\rangle$.
$n$ is at most $K$.
$\Xmax$ is the maximum of the random field $X(u)$, $u\in M$, in (\ref{Xu}) with
\[
 M = \bigl\{\varphi^{(i)}\in\S^{n-1} \mid 1\le i\le K \bigr\}.
\]
Define
\[
 Y_i=\sigma_i\langle\varphi^{(i)},\xi/\Vert\xi\Vert\rangle, \quad
 \Ymax =\max\{Y_1,\ldots,Y_K\}.
\]

The tube formula as well as its remainder order are given as a byproduct of the previous sections.

\begin{theorem}
\label{thm:bonferroni}
\begin{equation}
\label{PU4}
 \P\bigl(\Ymax > b\bigr) = \sum_{i=1}^K \frac{1}{2}
 \bar B_{\frac{1}{2},\frac{1}{2}(n-1)}(b^2/\sigma_i^2), \quad b\ge\bc,
\end{equation}
and
\begin{equation}
\label{bonferroni}
 \P\bigl(\Xmax > c\bigr) = \sum_{i=1}^K \bar\Phi(c/\sigma_i) + \bar G_n(c^2/\bc^2),
 \quad c\to\infty,
\end{equation}
where
\begin{equation}
\label{hij}
 \bc^2 = \max_{i\ne j}\frac{\sigma_i^2}{1+h(i,j)}, \quad h(i,j) =
 \frac{(\sigma_i/\sigma_j - \rho_{ij})_+^2}{1 - \rho_{ij}^2}.
\end{equation}
(\ref{PU4}) and (\ref{bonferroni}) hold even if the summation $\sum_{i=1}^K$ is replaced with $\sum_{i:\sigma_i>\bc}$.
\end{theorem}

The tube formula when the index set $M$ is finite is the Bonferroni method.
The next proposition corresponds to Proposition \ref{prop:validity}.
Because
$\bar\Phi(c/\sigma_i)\asymp c^{-1}e^{-c^2/2\sigma_i^2}$ and
$\bar G_n(c^2/\bc^2)\asymp c^{n-2}e^{-c^2/2\bc^2}$,
Proposition \ref{prop:bonferroni-validity} affirms that the approximation error in (\ref{bonferroni}) is asymptotically smaller than the leading term.

\begin{proposition}
\label{prop:bonferroni-validity}
The threshold $\bc$ in Theorem \ref{thm:bonferroni} satisfies
\begin{equation}
\label{bcri-bonferroni}
 \bc \le \max_{1\le i<j\le K}\sqrt{\frac{1+\rho_{ij}}{2}}\,\sigma_0<\sigma_0, \ \ \sigma_0=\max_{1\le i\le K}\sigma_i.
\end{equation}
\end{proposition}

The factor $\max_{i<j}\sqrt{(1+\rho_{ij})/2}$ in (\ref{bcri-bonferroni}) is $\bc=\cos\tc$ when $\sigma_i\equiv 1$ (see Remark \ref{rem:critical}), and is strictly less than 1.

\begin{example}
Suppose that $K=3$, $(\sigma_1,\sigma_2,\sigma_3)=(2,1,3)$, $\rho_{12}=\rho_{23}=1/\sqrt{2}$, $\rho_{13}=1/2$.
We have $\bc=3\sqrt{3/7}=1.964$, 
which is less than
\[
\max_{i<j}\sqrt{\frac{1+\rho_{ij}}{2}} \max_{i}\sigma_i=\sqrt{\frac{1+1/\sqrt{2}}{2}}\times 3 = 2.772 
\]
as proved in Proposition \ref{prop:bonferroni-validity}.
When $\sigma_i\le\bc$, $\P(Y_i>b)\equiv 0$ for $b\ge \bc$, which makes no contribution to $\P\bigl(\Ymax > b\bigr)$ in (\ref{PU4}).
In this example, $\sigma_2\le\bc$.
\end{example}

\section{Examples}
\label{sec:examples}

\subsection{Maximum of a Gaussian random process on a circle}

We identify $\Sym(2)$, the set of $2\times 2$ real-symmetric matrices, with $\R^3$ by the correspondence
\[
 z=(z_1,z_2,z_3)\ \leftrightarrow\ %
 Z = \begin{pmatrix}z_1 & z_2/\sqrt{2} \\ z_2/\sqrt{2} & z_3\end{pmatrix}, \qquad
\Vert z\Vert_{\R^3}^2 = \Vert Z\Vert_{\Sym(2)}^2 = \tr\bigl(Z^2\bigr).
\]
Let
\[
 M = \{ \varphi(t) = h(t) h(t)^\top \mid t \in [0,\pi) \}
 \subset \Sym(2), \quad h(t)=(\cos t,\sin t)^\top.
\]
Because $\Vert\varphi(t)\Vert_{\Sym(2)}^2=\tr\bigl((h(t)h(t)^\top)^2\bigr)=\Vert h(t)\Vert^4=1$ and $\varphi(0)=\varphi(\pi-)$,
$M$ is a closed curve in the unit sphere of $\Sym(2)$.
Indeed, $M$ is a circle with radius $1/\sqrt{2}$.
We define a Gaussian random process on $[0,\pi)$ by
\[
 X(t)
 = \sigma(t) \tr\bigl[h(t)h(t)^\top \Xi\bigr]
 = \sigma(t) h(t)^\top \Xi h(t), \quad
 \Xi=\begin{pmatrix}\xi_1 &\xi_2/\sqrt{2} \\ \xi_2/\sqrt{2} & \xi_3\end{pmatrix},
\]
where $\xi_1,\xi_2,\xi_3$ are i.i.d.\ random variables from $\mathcal{N}(0,1)$, and 
\[
 \sigma(t)=e^{\ell(t)}=e^{-m\sin^2 t}, \quad m\ge 0.
\]
Because
$\dot\varphi = \dd\varphi/\dd t = \dot h h^\top + h \dot h^\top$, $\dot h = \dd h/\dd t = (-\sin t,\cos t)^\top$,
the metric is $g=\Vert \dot\varphi\Vert^2=2$, and the volume element of $M$ is $\sqrt{2}\,\dd t$.
The first and second derivatives of $\ell=-m\sin^2 t$ are $\dot\ell=\dd\ell/\dd t=-m\sin 2 t$ and $\ddot\ell=\dd^2\ell/\dd t^2=-2m\cos 2 t$, respectively.
The maximum of the variance is $1$ attained at $t=0$.

The tail probability of the maximum $\Xmax=\max_{t\in[0,\pi)}X(t)$ is approximated by the tube method (\ref{PT2}) as
\[
 \P\bigl(\Xmax > c\bigr) =
 \int_0^\pi \frac{1+(-\ddot\ell+(\dot\ell)^2)/2}{2\pi\,(1+(\dot\ell)^2/2)}
 \,\bar G_2\biggl(\frac{1+(\dot\ell)^2/2}{\sigma^2}c^2\biggr) \,\sqrt{2}\,\dd t
 + O(\bar G_3\bigl(c^2/\bc^2)\bigr), \ \ c\to\infty,
\]
and its Laplace approximation is
\[
 \P\bigl(\Xmax > c\bigr) \sim \sqrt{\frac{1-\ddot\ell(0)/2}{-\ddot\ell(0)/2}}\,\bar\Phi(c)
 = \sqrt{\frac{1+m}{m}}\,\bar\Phi(c), \quad c\to\infty.
\]

The threshold $\bc$ is calculated as follows.
The bases of the tangent space and the normal space are given by
$\dot\varphi = \dot h h^\top + h \dot h^\top$ and $v = \dot h\dot h^\top$, respectively.
A simple calculation yields
\[
 \psi(t) = \varphi(t) - \frac{1}{2}\dot\ell(t)\dot\varphi(t),
\] 
\begin{align*}
\frac{\sigma(t)}{\sigma(s)} - \langle s,\psi(t)\rangle 
=& e^{\ell(t)-\ell(s)} - (h(t)^\top h(s))^2
 + \dot \ell(t) \dot h(t)^\top h(s) h(t)^\top h(s) \\
=& e^{-m(\sin^2 t - \sin^2 s) } - \cos^2(s-t)
 - m \sin(2t) \sin (s-t) \cos (s-t),
\end{align*}
\[
 \Vert P_t^\perp s\Vert = |v(t)^\top\varphi(s)| = (\dot h(t)^\top h(s))^2 = \sin^2(s-t),
\]
\begin{align*}
h(t,s)
=&
 \frac{1}{2} m^2\sin^2 (2t) \\
&
 + \biggl(
 \frac{\{e^{-m(\sin^2 t - \sin^2 s)}
 - \cos^2(s-t) - m \sin (2t) \sin(s-t) \cos(s-t)\}_+}{\sin^2(s-t)} \biggr)^2
 \biggr\}.
\end{align*}
Using these quantities, the threshold $\bc$ is obtained as
\begin{equation}
\label{bc}
 \bc^2 = \sup_{t\ne s}\frac{\sigma(t)^2}{1+h(t,s)}.
\end{equation}

The objective function to be maximized in (\ref{bc}) contains the truncation function $\{\cdot\}_+$.
The set of $t$ where the truncation is active is
\[
 \{ t \in [0,\pi) \mid \sigma(t)/\sigma(s) - \langle s,\psi(t)\rangle < 0,\ \exists s\in [0,\pi) \},
\]
which is $\emptyset$ ($0\le m<1$), $\{0.5\pi\}$ ($m=1$), $[0.431\pi,0.568\pi]$ ($m=1.1$), $[0.366\pi,0.633\pi]$ ($m=1.5$), $[0.266\pi,0.716\pi]$ ($m=10$).
Therefore, we should include the truncation when $m\ge 1$.
Taking this into observation, we conducted the maximization in (\ref{bc}).
Numerical calculation suggests that the supremum in (\ref{bc}) is attained when $t,s\to 0$.
That is, it is conjectured that
\begin{equation}
\label{conjecture}
 \bc^2
 = \limsup_{t,s\to 0,\,t\ne s} \frac{\sigma(t)^2}{1+h(t,s)}
 = \frac{1}{1 + (1+m)^2}.
\end{equation}
This is strictly less than $\max_{t\in M} \sigma(t)^2 = \sigma(0)^2 = 1$.

Figure \ref{fig:tailprob} shows approximations to the tail probability of $\Xmax$
using Monte Carlo simulations (10,000 replications), the tube method, and the Laplace approximation.
The tube method provides accurate approximations in all configurations,
but looks more accurate when $m$ is large.
This observation is consistent with (\ref{conjecture}).
The Laplace approximation works well when $m$ is large.

\begin{figure}
\begin{center}
\begin{tabular}{cc}
\scalebox{0.775}{\includegraphics{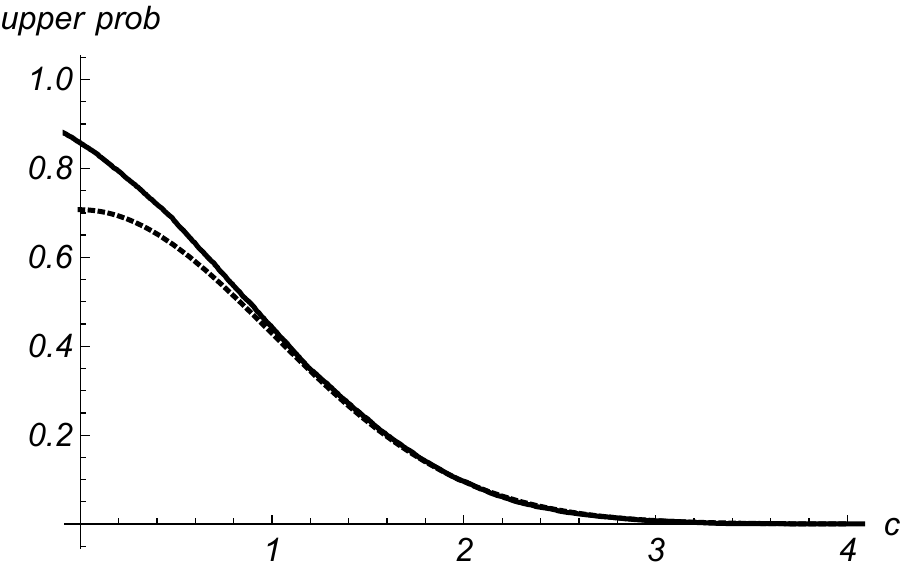}}
&
\scalebox{0.775}{\includegraphics{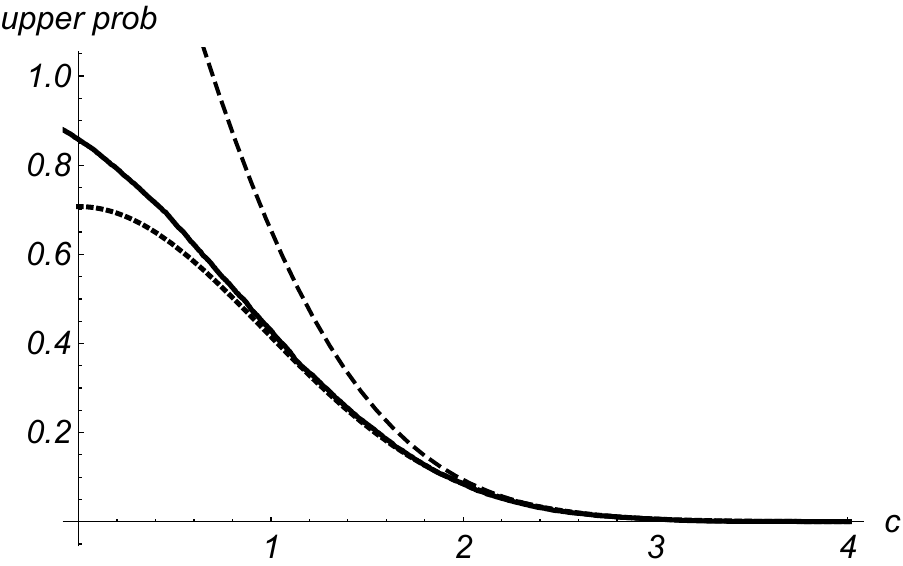}}
\\
{\small $m=0$} & {\small $m=1/16$} \\[3mm]
\scalebox{0.775}{\includegraphics{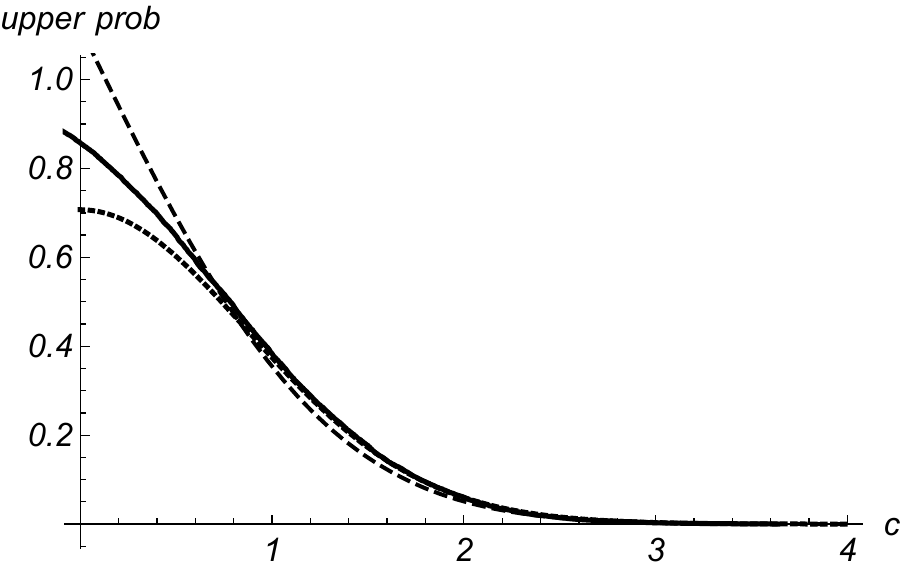}}
&
\scalebox{0.775}{\includegraphics{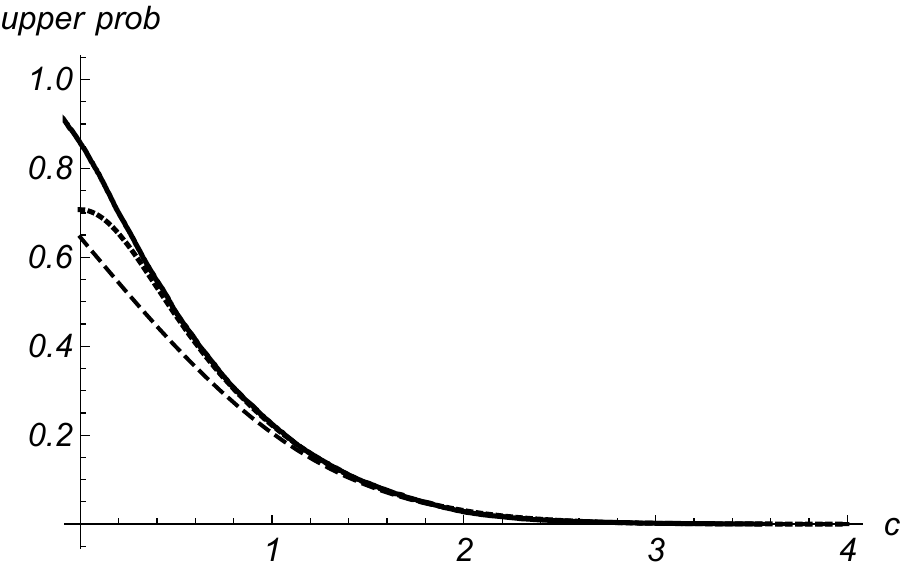}}
\\
{\small $m=1/4$} & {\small $m=3/2$}
\end{tabular}
\caption{Upper-tail probability of $\max_t X(t)$.\newline
\small
\textbf{---}: Simulation, $\bm{\cdots}$: Tube method, $\bm{- -}$: Laplace approximation.
}
\label{fig:tailprob}
\end{center}
\end{figure}

\subsection{Largest eigenvalue of a Wishart matrix: Laplace approximation}

The distribution of a $p\times p$ Wishart matrix with $\nu$ degrees of freedom and matrix parameter $\Lambda$ is denoted by $\mathcal{W}_p(\nu,\Lambda)$.
The largest eigenvalue of the Wishart matrix $\mathcal{W}_p(\nu,\Lambda)$ can be formulated as the maximum of a Gaussian random field.
Hence, its tail probability can be approximated by the tube method.
When the matrix parameter $\Lambda$ is the identity matrix, the tube formula and its approximation error are studied in \cite{kuriki-takemura:2001,kuriki-takemura:2008a}.
Here we derive the Laplace approximation formula for the tail probability when $\Lambda$ is not proportional to the identity matrix.

Without loss of generality, we assume
\[
 \Lambda=\diag(\lambda_1,\ldots,\lambda_p), \quad
 \lambda_1=\cdots=\lambda_q>\lambda_{q+1}\ge\cdots\ge\lambda_p> 0.
\]

The largest eigenvalue $\ell_1$ of the Wishart matrix $\mathcal{W}_p(\nu,\Lambda)$ is the square of the largest singular value $s_1$ of $\Lambda^{\frac{1}{2}}\Xi$, where $\Xi=(\xi_{ij})$ is a $p\times \nu$ random matrix consisting of i.i.d.\ standard Gaussian random variables.
From the definition of the largest singular value, we obtain
\begin{equation}
\label{s1}
 s_1 = \max_{\Vert v\Vert=\Vert w\Vert=1} v^\top \Lambda^{\frac{1}{2}}\Xi w
= \max_{u\in M} \sigma(u) \langle u,\vec(\Xi)\rangle,
\end{equation}
where
\begin{align}
\label{sigma-varphi-M}
& \sigma(u) = \sqrt{v^\top\Lambda v}, \qquad
  u=\varphi(v,w)=\frac{(\Lambda^{\frac{1}{2}}\otimes I_{\nu})(v\otimes w)}{\sqrt{v^\top \Lambda v}}, \\
& M = \{ \varphi(v,w)\in\R^{p\nu} \mid \Vert v\Vert=\Vert w\Vert=1 \}, \nonumber \\
& \vec(\Xi) = (\xi_{11},\xi_{21},\ldots,\xi_{p,1},\xi_{12},\ldots,\xi_{p,\nu})^\top. \nonumber
\end{align}
The maximizer of $\sigma(u)$ is
\[
 M_0 =\biggl\{ \varphi\biggl(\begin{pmatrix} v_1 \\ 0 \end{pmatrix},w\biggr) \in\R^{p\nu} \mid v_1\in\R^q,\,w\in\R^\nu,\,\Vert v_1\Vert=\Vert w\Vert=1 \biggr\}.
\]
The dimensions of $M$ and $M_0$ are
\begin{equation}
\label{dd0}
 d=\dim M=p+\nu-2, \quad d_0=\dim M_0=q+\nu-2.
\end{equation}

Fix a point $u_0=\varphi(v_0,w_0)\in M_0$,
$v_0=(1,0,\ldots,0)^\top$, $w_0=(1,0,\ldots,0)^\top$,
and take local coordinates of $M$ by
\[
\textstyle
 v=\Bigl(\sqrt{1-\sum_{i=1}^{p-1} (t^i)^2},t^1,\ldots,t^{p-1}\Bigr)^\top, \quad
 w=\Bigl(\sqrt{1-\sum_{i=p}^{p+\nu-2} (t^i)^2},t^p,\ldots,t^{p+\nu-2}\Bigr)^\top.
\]
Then, $u_0=\varphi(v,w)|_{t=0}$.

A simple calculation yields the metric tensor $G$ and the Hesse matrix $C$ of $\ell=\log\sigma$ evaluated at $u_0$ as
\begin{equation}
\label{GC}
 G = \begin{pmatrix}
 I_{q-1} & 0 & 0 \\
 0 & \frac{1}{\lambda_1}\Lambda_2 & 0 \\
 0 & 0 & I_{\nu-1} \end{pmatrix}, \quad
 C = \begin{pmatrix}
 0 & 0 & 0 \\
 0 & I_{p-q}-\frac{1}{\lambda_1}\Lambda_2 & 0 \\
 0 & 0 & 0 \end{pmatrix}, \quad
 \Lambda_2=\diag(\lambda_{q+1},\ldots,\lambda_p).
\end{equation}

We substitute (\ref{dd0}) and (\ref{GC}) into (\ref{laplace}).
Because the integrand of $\int_{M_0} \dd u_0$ is constant owing to symmetry, the integration is replaced by multiplication with the constant
\[
 \Vol_{d_0}(M_0)
= \frac{1}{2}\Vol_{d_0}\bigl(\S^{q-1}\times\S^{\nu-1}\bigr)
= \frac{\Omega_q\Omega_\nu}{\Omega_1}.
\]
Hence,
\[
 \P(s_1 > c) \sim
 \frac{1}{\Omega_{q+\nu-1}}
 \frac{\Omega_q\Omega_\nu}{\Omega_1}
 \frac{\det(I_{p-q})^{\frac{1}{2}}}{\det\bigl(I_{p-q}-\frac{1}{\lambda_1}\Lambda_2\bigr)^{\frac{1}{2}}}
 \,\bar G_{(q+\nu-2)+1}(c^2/\lambda_1).
\]
\begin{theorem}
\label{thm:wishart}
The tail probability of the largest eigenvalue $\ell_1$ of the Wishart matrix
$\mathcal{W}_p(\nu,\Lambda)$, where $\Lambda$ has the eigenvalues $\lambda_1=\cdots=\lambda_q>\lambda_{q+1}\ge\cdots\ge\lambda_p>0$, is
\[
 \P(\ell_1 > c) \sim
  \frac{1}{\prod_{i=q+1}^p(1-\lambda_i/\lambda_1)^{1/2}}
 \frac{\Omega_q\Omega_\nu}{\Omega_{q+\nu-1}\Omega_1}
 \,\bar G_{q+\nu-1}(c/\lambda_1), \quad c\to\infty.
\]
\end{theorem}

\begin{remark}
When $\nu=1$, Theorem \ref{thm:wishart} reduces to the formula for the weighted sum of the chi-square random variables:
\[
 \P\Biggl(\sum_{i=1}^p \lambda_i \xi_i^2 > c\Biggr) \sim 
 \frac{1}{\prod_{i=q+1}^p(1-\lambda_i/\lambda_1)^{1/2}}
 \,\bar G_{q}(c/\lambda_1), \quad c\to\infty,
\]
where $\xi_i^2\sim\chi^2_1$ i.i.d.\ (\cite[Theorem 2]{beran:1975}).
For exact calculation, see \cite{koyama-takemura:2016} and references therein.
\end{remark}

\subsection{Largest eigenvalue of a $2\times 2$ Wishart matrix}

When the size of Wishart matrix is small, we can derive the tube method approximation without using Laplace's method.
We consider the Wishart distribution $\mathcal{W}_2(\nu,\Lambda)$, $\Lambda=\diag(\lambda_1,\lambda_2)$ ($\lambda_1\ge\lambda_2$).
Define $\sigma(u)$, $\varphi(v,w)$ and $M$ as in (\ref{sigma-varphi-M}) with $p=2$, and consider the distribution $s_1$ in (\ref{s1}).

Fix a point $u_0=\varphi(v_0,w_0)\in M$,
$v_0=\bigl(\cos\theta,\sin\theta\bigr)^\top$,
$w_0=(1,0,\ldots,0)^\top$,
and take local coordinates of $M$ about $u$ by
\[
\textstyle
 v=\bigl(\cos t_1,\sin t_1\bigr)^\top, \quad
 w=\Bigl(\sqrt{1-\sum_{i=2}^{\nu} t_i^2},t_2,\ldots,t_{\nu}\Bigr)^\top.
\]
Then, $u_0=\varphi(v,w)|_{t_1=\theta,t_2=\cdots=t_\nu=0}$.

A simple calculation yields
\begin{equation}
\label{q1}
\sigma^2 = \lambda_1\cos^2\theta + \lambda_2\sin^2\theta, \qquad
q = \frac{1+\Vert\nabla\ell\Vert^2}{\sigma^2} = \lambda_1^{-1}\cos^2\theta+\lambda_2^{-1}\sin^2\theta,
\end{equation}
and
\[
 G = (g_{ij}) = \diag\bigl(g_{11},\underbrace{1,\ldots,1}_{\nu-1}\bigr), \qquad 
  g_{11} = \frac{\lambda_1\lambda_2}{\sigma^4}, \qquad
 \G = (\g_{ij}) = G + C = I_{\nu}.
\]
The volume element of $M$ is
\begin{equation}
\label{vol_element}
 \sqrt{g_{11}}\dd\theta\,\dd w =
 \frac{\sqrt{\lambda_1\lambda_2}}{\sigma^2}\dd\theta\,\dd w,
\end{equation}
where $\dd w$ is the volume element of $\S^{\nu-1}$, and
\begin{equation}
\label{CG}
 \det(I_{\nu}+CG^{-1}) = \det(G+C)\det(G)^{-1} = g_{11}^{-1} =
 \frac{\sigma^4}{\lambda_1\lambda_2}.
\end{equation}

According to (\ref{R}), the curvature tensor is
\[
 R_{ij;kl} = \begin{cases}
 \delta_{ik}\delta_{jl}-\delta_{il}\delta_{jk} & (\mbox{if $i,j,k,l\ge 2$}), \\
 0 & (\mbox{otherwise}), \end{cases}
\]
where $\delta_{ij}$ is the Kronecker delta.
Suppose that $i<j$ and $k<l$.
If $i,k\ge 2$, then both $R_{ij;kl}$ and $g_{ik}g_{jl}-g_{il}g_{jk}$ are equal to $\delta_{ik}\delta_{jl}-\delta_{il}\delta_{jk}$, and hence $R_{ij;kl} - (g_{ik}g_{jl}-g_{il}g_{jk})$ vanishes.
If $i=1<k$ or $k=1<i$, then $R_{ij;kl}=0$ and $g_{ik}g_{jl}-g_{il}g_{jk}=0$.
Remaining case is
\[
 R_{ij;kl} - (g_{ik}g_{jl}-g_{il}g_{jk}) =
 -g_{11} \quad \mbox{when } i=k=1<j=l.
\]
Therefore, the function $\zeta_e$ in Theorem \ref{thm:volume} is
\begin{equation}
\label{zeta1}
 \zeta_e =
\begin{cases}
 1 & (e=0), \\
 -(\nu-1) g_{11} & (e=2), \\
 0 & (\mbox{otherwise}).
\end{cases}
\end{equation}
Substituting (\ref{q1}), (\ref{vol_element}), (\ref{CG}), and (\ref{zeta1}) into (\ref{PT2}), we have the approximation formula for the largest eigenvalue $\ell_1$ as
\begin{align}
\label{Ptube}
\Ptube(\ell_1 > c)
=& \Ptube(s_1 > \sqrt{c}) \\
=& \frac{1}{2}\int_{0}^{2\pi}\int_{\S^{\nu-1}}
 \frac{\sqrt{\lambda_1\lambda_2}}{\sigma^2}\dd\theta\,\dd w
 \frac{1}{\Omega_{\nu+1}}
 \frac{\sigma^4}{\lambda_1\lambda_2}
 (\sigma^2 q)^{-\frac{1}{2}(\nu+1)} \bar G_{\nu+1}(q c) \nonumber \\
& + \frac{1}{2}\int_{0}^{2\pi}\int_{\S^{\nu-1}}
 \frac{\sqrt{\lambda_1\lambda_2}}{\sigma^2}\dd\theta\,\dd w
 \frac{1}{2\pi\Omega_{\nu-1}}
  \frac{\sigma^4}{\lambda_1\lambda_2}
 (\sigma^2 q)^{-\frac{1}{2}(\nu-1)} \bar G_{\nu-1}(q c) \nonumber \\
& \qquad\qquad
 \times\biggl(-(\nu-1)\frac{\lambda_1\lambda_2}{\sigma^4}\biggr) \nonumber \\
=& \frac{\Omega_{\nu}}{2\Omega_{\nu+1}}
 \int_{0}^{2\pi}
   (\sigma^2 q)^{-\frac{1}{2}(\nu+1)} \biggl[ \frac{\sigma^2}{\sqrt{\lambda_1\lambda_2}}\bar G_{\nu+1}(q c)
 - \sqrt{\lambda_1\lambda_2}q \bar G_{\nu-1}(q c) \biggr] \dd\theta.
\nonumber
\end{align}

Next, we obtain the threshold $\bc$.
We prepare a lemma for this purpose.
\begin{lemma}
\label{lem:wishart-critical}
For $a>0$,
\[
 \inf_{|y|<1}\frac{(a-b y)_+^2}{1-y^2} = (a^2-b^2)_+.
\]
\end{lemma}
\begin{proof}
When $|b|\le a$, $a-b y\ge 0$ for all $|y|<1$, and the left-hand side is
\[
 \inf_{|y|<1}\frac{(a-b y)^2}{1-y^2} = \frac{(a-b y)^2}{1-y^2}\Big|_{y=b/a} = a^2-b^2.
\]
When $|b|>a$, $a-b y=0$ at $-1<y=a/b<1$, and the left-hand side is 0.
\end{proof}

As in (\ref{sigma-varphi-M}), let $u=\varphi(v,w)$ and $\uu=\varphi(\vv,\ww)$, where
\[
  \varphi(v,w)=\frac{(\Lambda^{\frac{1}{2}}\otimes I_{\nu})(v\otimes w)}{\sigma(u)}, \quad \sigma(u)=\sqrt{v^\top \Lambda v}.
\]
We will evaluate
\[
 1/\bc^2 = \inf_{u\ne\uu}\frac{1}{\sigma(u)^2}\biggl\{
 1+\Vert\nabla\ell(u)\Vert^2 + \biggl( \frac{\{ \sigma(u)/\sigma(\uu)
   - \langle \uu,u-\nabla\ell(u)\rangle\}_+}{\Vert P_u^\perp \uu \Vert}\biggr)^2 \biggr\}.
\]
Let $v = (\cos t,\sin t)^\top$, $\vv = (\cos s,\sin s)^\top$.
Then, simple calculations yield
\[
 \sigma(u)=\sqrt{\lambda_1\cos^2 t+\lambda_2\sin^2 t}, \quad
 \sigma(\uu)=\sqrt{\lambda_1\cos^2 s+\lambda_2\sin^2 s},
\]
\begin{align*}
 & \langle \uu,u\rangle = \frac{\lambda_1\cos t\cos s+\lambda_2\sin t\sin s}{\sigma(u)\sigma(\uu)}(\ww^\top w), \\
 & \langle \uu,\nabla\ell(u)\rangle = \frac{(\lambda_1-\lambda_2)\sin(2 t)\sin(t-s)}{2\sigma(u)\sigma(\uu)}(\ww^\top w),
\end{align*}
\begin{align*}
& \Vert P^\perp_{u}\uu\Vert = \frac{\sqrt{\lambda_1\lambda_2}|\sin(t-s)|}{\sigma(u)\sigma(\uu)}\sqrt{1-(\ww^\top w)^2}, \qquad
\Vert\nabla\ell(u)\Vert^2 = \frac{(\lambda_1-\lambda_2)^2\sin^2(2 t)}{4\lambda_1\lambda_2},
\end{align*}
and hence,
\begin{align*}
& \frac{\{ \sigma(u)/\sigma(\uu)
   - \langle \uu,\psi(u)\rangle\}_+^2}{\Vert P_u^\perp \uu \Vert^2} \\
& = \frac{\{\sigma(u)^2 -[(\lambda_1\cos t\cos s+\lambda_2\sin t\sin s)-
  \frac{1}{2}(\lambda_1-\lambda_2)\sin(2 t)\sin(t-s)](\ww^\top w)\}_+^2}{\lambda_1\lambda_2\sin^2(t-s)(1-(\ww^\top w)^2)}.
\end{align*}
By means of Lemma \ref{lem:wishart-critical},
\begin{align*}
& \inf_{w\ne\ww}\frac{\{ \sigma(u)/\sigma(\uu)
   - \langle \uu,\psi(u)\rangle\}_+^2}{\Vert P_u^\perp \uu \Vert^2} \\
& = \frac{\{\sigma(u)^4 -[(\lambda_1\cos t\cos s+\lambda_2\sin t\sin s)-
  \frac{1}{2}(\lambda_1-\lambda_2)\sin(2 t)\sin(t-s)]^2\}_+}{\lambda_1\lambda_2\sin^2(t-s)} \\
& = \frac{\{\sigma(u)^4\sin^2(t-s)\}_+}{\lambda_1\lambda_2\sin^2(t-s)} = \frac{\sigma(u)^4}{\lambda_1\lambda_2},
\end{align*}
which is independent of $\vv$.
Therefore, we have the threshold $\bc$ as follows.
\[
 1/\bc^2 = \inf_{u}\frac{1}{\sigma(u)^2}\biggl\{
 1+\frac{(\lambda_1-\lambda_2)^2\sin^2(2 t)}{4\lambda_1\lambda_2}+\frac{\sigma(u)^4}{\lambda_1\lambda_2}\biggr\}
 = \inf_{u}\biggl(\frac{1}{\lambda_1}+\frac{1}{\lambda_2}\biggr)
 = \frac{1}{\lambda_1}+\frac{1}{\lambda_2}.
\]
The results of this subsection are summarized as follows.
\begin{theorem}
The upper probability of the largest eigenvalue $\ell_1$ of the $2\times 2$ Wishart matrix $\mathcal{W}_2(\nu,\Lambda)$, $\Lambda=\diag(\lambda_1,\lambda_2)$, is
\[
 \P(\ell_1>c) = \Ptube(\ell_1>c) + O(\bar G_{2\nu}(c/\bc^2)), \ \ c\to\infty,
\]
where
$\Ptube(\ell_1>c)$ is given in (\ref{Ptube}), and
\[
 \bc = \sqrt{\frac{\lambda_1\lambda_2}{\lambda_1+\lambda_2}}.
\]
\end{theorem}

Figure \ref{fig:wishart} shows the tail probability of the largest eigenvalue of the $2\times 2$ Wishart matrix $\mathcal{W}_2(4,\Lambda)$ for various $\Lambda$'s.
The simulation is based on 10,000 replications.
The tube formula shows a good performance for estimating the tail probability of the regions used in testing hypotheses.
The Laplace approximation also works well when largest eigenvalue is dominating.
This is consistent with the fact that when $\lambda_1=1$, $\bc = \sqrt{1+1/\lambda_2}$ is a decreasing function in $\lambda_2$.

\begin{figure}
\begin{center}
\begin{tabular}{cc}
\scalebox{0.775}{\includegraphics{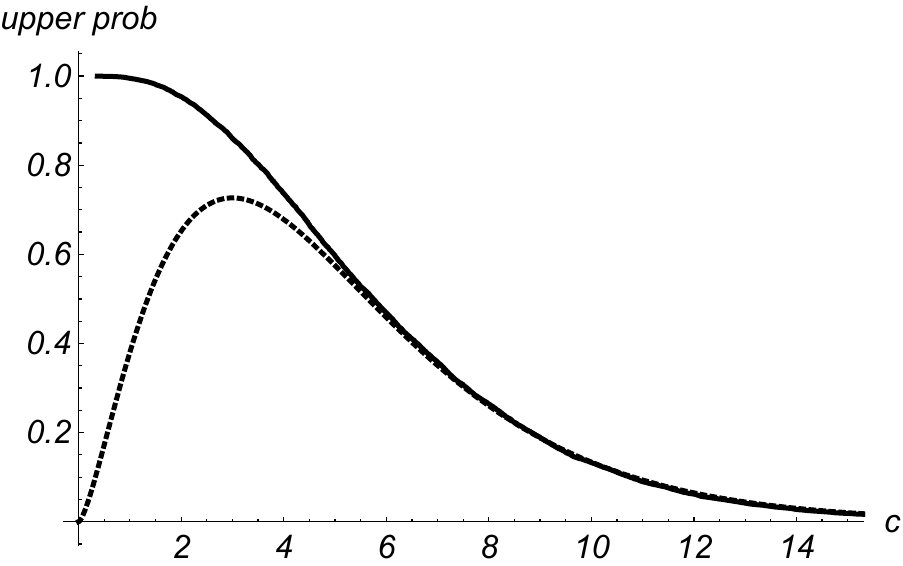}}
&
\scalebox{0.775}{\includegraphics{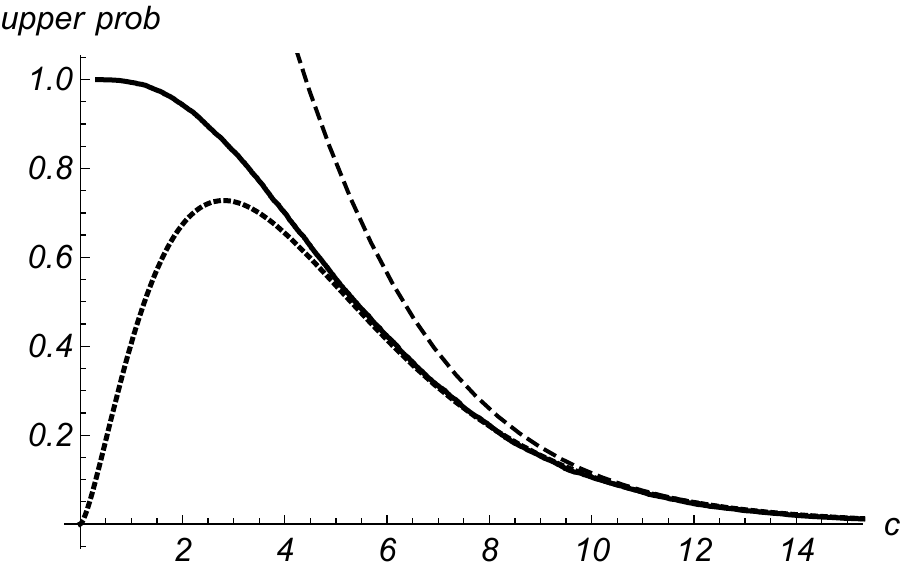}}
\\
{\small $\Lambda=\diag(1,1)$} & {\small $\Lambda=\diag(1,7/8)$}
\\[3mm]
\scalebox{0.775}{\includegraphics{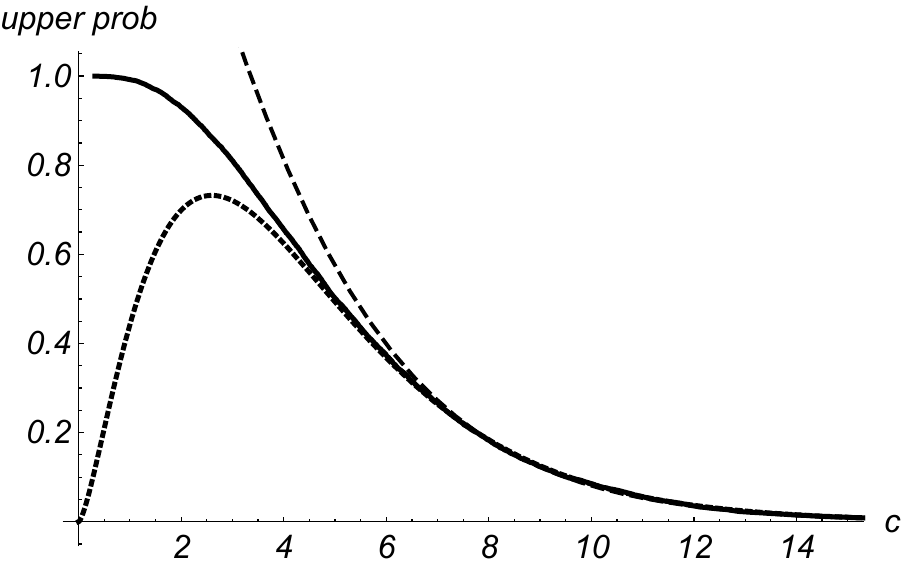}}
&
\scalebox{0.775}{\includegraphics{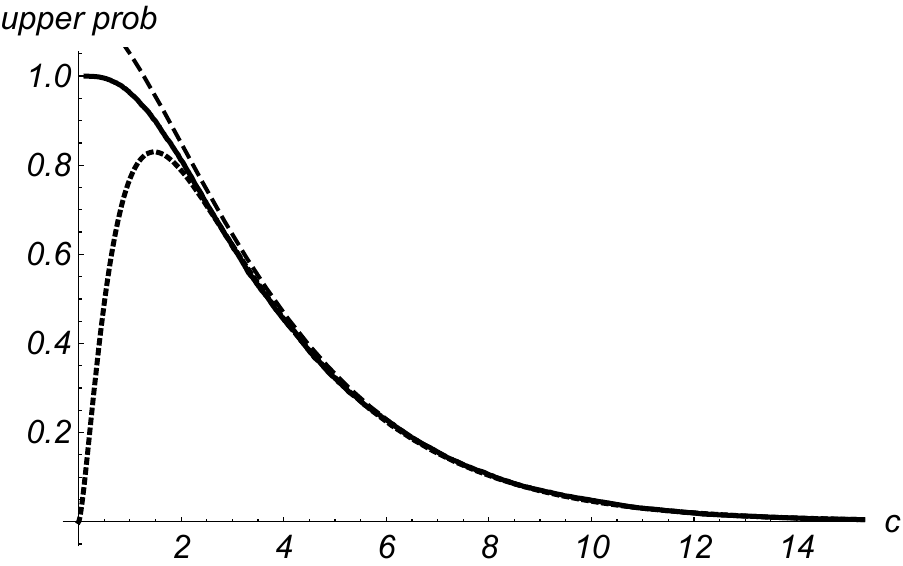}}
\\
{\small $\Lambda=\diag(1,3/4)$} & {\small $\Lambda=\diag(1,1/4)$}
\end{tabular}
\caption{Upper-tail probability of the largest eigenvalue of $\mathcal{W}_2(4,\Lambda)$.\newline
\small
\textbf{---}: Simulation, $\bm{\cdots}$: Tube method, $\bm{- -}$: Laplace approximation.
}
\label{fig:wishart}
\end{center}
\end{figure}

\section{Proofs}
\label{sec:proofs}

In this section we prove theorems and propositions of Section \ref{sec:main}.
We represent elements of $\R^n$ as column vectors.
The inner product of $u, w\in\R^n$ is denoted by $\langle u,w\rangle$ or $u^\top w$ interchangeably.

\subsection{Notation}

Recall that the point of $M$ is written as $u=\varphi(t)$, $t=(t^i)_{1\le i\le d}$.
Write $\partial/\partial t^i=\partial_i$, and
$\varphi_i=\partial_i\varphi$,
$\varphi_{ij}=\partial_i\partial_j\varphi$, $\sigma_i=\partial_i\sigma$, \textit{etc}.
The metric tensor and the connection coefficient are
$g_{ij} = \langle \varphi_i,\varphi_j\rangle$ and
\[
 \Gamma_{ij,k}
 = \langle\varphi_{ij},\varphi_k\rangle
 = \frac{1}{2}(\partial_i g_{jk} + \partial_j g_{ik} - \partial_k g_{ij}),
\]
respectively.
Note that $\partial_k g_{ij}=\Gamma_{ik,j}+\Gamma_{jk,i}$.

The tangent space of $\S^{n-1}=S(\R^n)$ and $M$ at $u=\varphi$ are written as
\[
 T_u(\S^{n-1}) = \{w \in\R^n \mid u^\top w=0 \}, \qquad
 T_u(M) = \spn\{\varphi_i \mid 1\le i\le d\},
\]
respectively.

An orthonormal basis of the normal space $N_u$ in (\ref{Nu}) is denoted by
$n_a=n_a(u)$, $1\le a\le n-d-1$.
That is,
\[
 N_u = T_u(M)^\perp\cap T_u(\S^{n-1}) = \spn\{n_a \mid 1\le a\le n-d-1\}.
\]
Note that
\begin{equation}
\label{osum}
 \R^n = T_u(\R^n) = \spn\{u\}\oplus T_u(M) \oplus N_u.
\end{equation}
Throughout this section, we use Einstein's convention that $a_i b^i$ means $\sum_i a_i b^i$ unless it causes confusion.

\subsection{Jacobian}

Let $z\in\R^n$ be fixed.
The maximal point $u\in M$ of the function $u \mapsto \sigma(u)\langle u,z\rangle$ satisfies the ``normal equation'':
\begin{align}
\label{normal}
0
= \partial_i (\sigma \langle \varphi,z\rangle)
= \langle \sigma_i \varphi + \sigma \varphi_i,z\rangle, \quad 1\le i\le d.
\end{align}
Because of (\ref{osum}), any $z\in\R^n$ is represented by $z = r \varphi + r^i \varphi_i + s^a n_a$.
Then, (\ref{normal}) becomes $0=\sigma_i r + \sigma r^j g_{ij}$, and hence
$r^j = - \sigma^{-1} \sigma_i r g^{ij} = - \ell_i r g^{ij}$,
where $\ell_i=\partial_i\ell=\partial_i\log\sigma$ and $g^{ij}$ is the $(i,j)$th element of $(g_{ij})^{-1}$.
Therefore, at point $u$ such that $\max_u X(u) = X(u)$,
\begin{equation}
\label{one-to-one}
 z = r \varphi - \ell_i r g^{ij} \varphi_j +  s^a n_a = r \psi +  s^a n_a,
\end{equation}
where
\[
 \psi = \varphi - \ell_i g^{ij} \varphi_j = u - \nabla\ell, \quad \nabla\ell=\ell_i g^{ij} \varphi_j.
\]
$\nabla\ell$ is the gradient of $\ell$ in the tangent space $T_u(M)$.
This defines a local one-to-one transformation between
$z\leftrightarrow (r,t^i,s^a)$.
We derive the Jacobian of this transformation as follows.

Taking the derivatives of (\ref{one-to-one}), we obtain
\[
 \dd z = \dd r \psi + r \dd\psi + \dd s^a n_a + s^a \dd n_a,
\]
where
\begin{align*}
\dd\psi
=& \dd\varphi - \dd[\ell_j g^{ji}\varphi_i]
= \varphi_k \dd t^k - \partial_k [\ell_j g^{ji}\varphi_i] \dd t^k \\
=& [\varphi_k - \ell_{kj} g^{ji}\varphi_i
 -\ell_j (\partial_k g^{ji})\varphi_i
 -\ell_j g^{ji} \varphi_{ki}] \dd t^k, \\
\dd n_a =& (\partial_k n_a) \dd t^k.
\end{align*}
We will evaluate
$\varphi^\top \dd z$,
$\varphi_l^\top \dd z$,
$n_b^\top \dd z$, in turn.

(i) To obtain $\varphi^\top \dd z$, we have
\begin{align*}
\varphi^\top\psi =& 1, \\
\varphi^\top \dd\psi =&
  -\ell_j g^{ji}\varphi^\top \varphi_{ki} \dd t^k
 = \ell_j g^{ji}g_{ki} \dd t^k
 = \ell_k \dd t^k, \\
\varphi^\top \dd n_a =& (\varphi^\top\partial_k n_a) \dd t^k = -(\varphi_k^\top n_a) \dd t^k = 0.
\end{align*}
Summarizing the above,
\begin{equation}
 \varphi^\top \dd z = \dd r + r\ell_k \dd t^k.
\label{varphi-dz}
\end{equation}

(ii) To obtain $\varphi_l^\top \dd z$, we have
\begin{align*}
\varphi_l^\top\psi
=& - \ell_j g_{li} g^{ji} = -\ell_l, \\
\varphi_l^\top \dd\psi
=& [g_{lk} - \ell_{kj} g^{ji}g_{il}
  - \ell_j (\partial_k g^{ji})g_{li}
  - \ell_j g^{ji}\varphi_l^\top \varphi_{ik}] \dd t^k \\
=& [g_{lk} - \ell_{kl}
  + \ell_j g^{ji}(\partial_k g_{li})
  - \ell_j g^{ji}\varphi_l^\top \varphi_{ik}] \dd t^k \\
=& [g_{lk} - \ell_{kl}
  + \ell_j g^{ji}(\Gamma_{lk,i}+\Gamma_{ik,l})
  - \ell_j g^{ji}\Gamma_{ik,l}] \dd t^k \\
=& [g_{lk} - (\ell_{kl} - \ell_j g^{ji}\Gamma_{lk,i})] \dd t^k
= [g_{lk} - \nabla_k\nabla_l\ell] \dd t^k,
\end{align*}
where
\[
 \nabla_l\ell = \partial_l\ell, \quad
 \nabla_k\nabla_l\ell = \partial_k\partial_l\ell -\Gamma_{kl,i}g^{ij}\ell_j
\]
are the covariant differentials.
Moreover,
\begin{align*}
\varphi_l^\top \dd n_a =& \varphi_l^\top\partial_k n_a \dd t^k
= -\varphi_{lk}^\top n_a \dd t^k = h_{kl,a} \dd t^k,
\end{align*}
where we let $h_{kl,a}=-\varphi_{lk}^\top n_a$.
Summarizing the above, 
\begin{equation}
\varphi_l^\top \dd z = - \ell_l \dd r
 + r[g_{lk}-\nabla_l\nabla_k\ell] \dd t^k
 + s^a h_{kl,a} \dd t^k.
\label{varphi_l-dz}
\end{equation}

(iii) To obtain $n_b^\top dz$, we have
$n_b^\top\psi = 0$, and
$n_b^\top \dd\psi$ and $n_b^\top \dd n_a$ consist of terms containing $\dd t^k$.
Therefore,
\begin{equation}
 n_b^\top \dd z = (\mbox{terms containing }\dd t^k) + \delta_{ab} \dd s^a,
\label{n_b-dz}
\end{equation}
where $\delta_{ab}$ is the Kronecker delta.

Combining (\ref{varphi-dz}), (\ref{varphi_l-dz}), and (\ref{n_b-dz}), we obtain
\begin{equation}
 \begin{pmatrix}
  \varphi^\top \\ \varphi_l^\top \\ n_b^\top
 \end{pmatrix} \dd z =
 \begin{pmatrix}
  1 & r\nabla_k\ell & 0 \\
   -\nabla_l\ell & r[g_{lk}-\nabla_k\nabla_l\ell] +s^a h_{kl,a} & 0 \\
  0 & * & \delta_{ab}
 \end{pmatrix}
 \begin{pmatrix} \dd r \\ \dd t^k \\ \dd s^a \end{pmatrix},
\label{det}
\end{equation}
where entries marked by $*$ do not affect the following calculations.
The determinant of the matrix on the right-hand side of (\ref{det}) is
\[
 \det(r[g_{lk}-\nabla_k\nabla_l\ell
  +\nabla_k\ell\nabla_l\ell]
  +s^a h_{kl,a})
 = \det(r(G+C) + H(s)),
\]
where
$c_{kl} = - \nabla_k \nabla_l\ell + \nabla_k\ell \nabla_l\ell$,
$C =(c_{kl})$, $G =(g_{kl})$, and $H(s) =(s^a h_{kl,a})$ as before.
$H(s)$ is the second fundamental form in the normal direction $s^a n_a\in N_u$ in (\ref{Nu}).
When needed, we write $C=C_u$, $G=G_u$, and $H(s)=H_u(s)$ to indicate the point $u$ where these quantities are defined.

On the other hand,
\[
 \det(\varphi,\varphi_l,n_b) = \det(G)^{\frac{1}{2}}.
\]
Taking the wedge product of all rows in (\ref{det}),
\[
 \det(G)^{\frac{1}{2}} \bigwedge_{i=1}^n \dd z_i = r^d \dd r\,\det(G+C + r^{-1} H(s))\,\bigwedge_{i=1}^d \dd t^i\,\bigwedge_{a=1}^{n-d-1}\dd s^a.
\]

Moreover, let
\[
\textstyle
 s = \sqrt{\sum_{a=1}^{n-d-1}(s^a)^2}, \qquad
 v = \Bigl(\sum_{a=1}^{n-d-1} s^a n_a\Bigr)\big/s \in S(N_u),
\]
where $S(N_u)$ is the unit sphere of the linear space $N_u$.
We conclude that the local one-to-one correspondence
\[
 z = r \psi + s v, \quad \psi = u - \nabla \ell
\]
has the Jacobian
\[
 \dd z = \bigl|\det(I_d+ C G^{-1}) \det(I_d+ (s/r) H(v)(G+C)^{-1})\bigr|\, r^d \dd r \, \dd u \,s^{n-d-2} \dd s \, \dd v,
\]
where $\dd z$ is the $n$-dimensional Lebesgue measure at $z$,
$\dd u = \det(G)^{\frac{1}{2}} \prod_{i=1}^d \dd t^i$ is the volume element of $M$ at $u$, and $\dd v$ is the volume element of $S(N_u)$.
Note that
\begin{equation}
\label{note}
 \Vert z\Vert^2 = r^2(1+\Vert\nabla\ell\Vert^2) + s^2, \qquad
 \sigma(u)\langle u,z/\Vert z\Vert\rangle
 = \frac{\sigma(u)}{\sqrt{1+\Vert\nabla\ell\Vert^2+(s/r)^2}}.
\end{equation}

\subsection{$\bc$ and critical radius $\tc$}
\label{subsec:proof-critical}

In this subsection, we prove Theorem \ref{thm:critical} and Propositions \ref{prop:validity} and \ref{prop:supporting}.

In the previous subsection, we proved that
\begin{eqnarray*}
&& \mbox{$u\in M$ annihilates the gradient of map $u\mapsto\sigma(u)\langle u,z\rangle$} \\
&\iff& z\in\{r\psi(u) + s v \mid (r,s,v)\in Q_u\},
\end{eqnarray*}
where
\[
 Q_u = \{(r,s,v) \mid r>0,\ s\ge 0,\ v\in S(N_u)\}.
\]
Here, we will find additional conditions on $z$ (the set $V_u$) such that
\begin{eqnarray*}
&& \mbox{$u\in M$ is the unique maximal point of the map $u\mapsto\sigma(u)\langle u,z\rangle$} \\
&\iff& z\in\{r\psi(u) + s v \mid (r,s,v)\in Q_u \}\cap V_u.
\end{eqnarray*}
Once the set $V_u$ is identified, and the event that multiple $u$'s have the same maximal values has a probability of zero, then the upper probability of $\Ymax$ is written as
\begin{equation}
\label{PU2}
 \P\bigl(\Ymax > b\bigr)
= \int \1_{\bigl\{\sigma(u)\langle u,z/\Vert z\Vert\rangle > b,\,z\in V_u\bigr\}} |p(u,r,s,v)| \,\dd u \,\dd r \,\dd s \,\dd v,
\end{equation}
where
\begin{align}
\label{p}
p(u,r,s,v)
= \frac{1}{(2\pi)^{n/2}}
 e^{-\frac{1}{2}\{r^2(1+\Vert\nabla\ell(u)\Vert^2) + s^2\}}
 \det(I_d+ C_u G_u^{-1}) \\
 \times r^d s^{n-d-2} \det(I_d+ (s/r) H_u(v)(G_u+C_u)^{-1}).
\nonumber
\end{align}

\begin{proof}[Proof of Theorem \ref{thm:critical}]
From the definition, we write
\[
 V_u = \bigl\{ z \in \R^n \mid \sigma(w)\langle w,z\rangle
 < \sigma(u)\langle u,z\rangle,\ z = r \psi(u) + s v,
 \ \forall w\in M\setminus\{u\} \bigr\}.
\]
Then, $V_u\cap V_{u'}=\emptyset$, $u\ne u'$, and 
the complement $\R^n\setminus\bigcup_{u\in M} V_u$ has the Lebesgue measure zero
by Sard's theorem.

The restriction $z\in V_u$ is to prevent the overlap of the integration (\ref{PU2}).
If $b$ is sufficiently large, we do not need to impose this restriction.
For fixed $u$, we find a threshold of $b$, say $\bc'(u)$, as the infimum of $b$ such that
\begin{align*}
& \{(r,s,v)\in Q_u \mid \sigma(u)\langle u,z/\Vert z\Vert\rangle > b,\,z=r\psi(u)+sv\in V_u\} \\
& = \{(r,s,v)\in Q_u \mid \sigma(u)\langle u,z/\Vert z\Vert\rangle > b,\,z=r\psi(u)+sv\}\ \ \mbox{a.s.},
\end{align*}
or equivalently
\begin{equation}
\label{rsv}
 A_u(b) \subset B_u \ \ \mbox{a.s.},
\end{equation}
where
\begin{align*}
A_u(b) =&\, \{(r,s,v)\in Q_u \mid \sigma(u)\langle u,z/\Vert z\Vert\rangle > b,\,z=r\psi(u)+sv\}, \nonumber \\ 
B_u    =&\, \{(r,s,v)\in Q_u \mid z=r\psi(u)+sv\in V_u\}.
\end{align*}
The threshold $\bc$ in Theorem \ref{thm:critical} is obtained by $\bc=\sup_{u\in M}\bc'(u)$.

Because of (\ref{note}), the set $A_u(b)$ is
\begin{align}
\label{iii}
 A_u(b)
=&\,\biggl\{ (r,s,v)\in Q_u \mid 
 \frac{\sigma(u)}{\sqrt{1+\Vert\nabla\ell(u)\Vert^2+(s/r)^2}} > b \biggr\} \\
=&\,\bigl\{ (r,s,v)\in Q_u \mid (s/r)^2 < \sigma(u)^2/b^2-(1+\Vert\nabla\ell(u)\Vert^2) \bigr\}.
\nonumber
\end{align}

Next, we characterize the set $B_u$.
Note first that
\begin{eqnarray}
\label{Vu}
&&
r\psi(u)+sv\in V_u \\
&\iff& \sigma(w)\langle w,r\psi(u)+sv\rangle <
 \sigma(u)\langle u,r\psi(u)+sv\rangle,
 \ \ \forall w\in M\setminus\{u\} \nonumber \\
&\iff&
 \frac{s}{r} \langle w,v\rangle <
 \sigma(u)/\sigma(w)
   - \langle w,\psi(u)\rangle,\ \ \forall w\in M\setminus\{u\}.
  \nonumber
\end{eqnarray}
$M\setminus\{u\}$ is partitioned as $M_u^+\sqcup M_u^-$, where
\begin{align*}
 M_u^+ =& \{ w\in M \mid \sigma(u)/\sigma(w)
   - \langle w,\psi(u)\rangle > 0 \}, \\
 M_u^- =& \{ w\in M \mid \sigma(u)/\sigma(w)
   - \langle w,\psi(u)\rangle \le 0 \} \setminus\{u\}.
\end{align*}
We consider two cases $M_u^-=\emptyset$ and $M_u^-\ne\emptyset$, separately.

(i) When $M_u^-=\emptyset$, 
\begin{align}
\label{Bu1}
B_u
=&\,
\biggl\{ (r,s,v)\in Q_u \mid \mbox{``if }\langle w,v\rangle > 0, \mbox{ then }\ %
 \frac{s}{r} < \frac{\sigma(u)/\sigma(w) - \langle w,\psi(u)\rangle}
            {\langle w,v\rangle}\mbox{''},\ \ \forall w\in M\setminus\{u\} \biggr\} \\
=&\,
\biggl\{ (r,s,v)\in Q_u \mid 
 \frac{s}{r} <
 \inf_{w\in M\setminus\{u\}:\langle w,v\rangle > 0}\frac{\sigma(u)/\sigma(w) - \langle w,\psi(u)\rangle}{\langle w,v\rangle} \biggr\}.
\nonumber
\end{align}
Therefore, $A_u(b)\subset B_u$ iff
\begin{align*}
 \sqrt{\frac{\sigma(u)^2}{b^2}-(1+\Vert\nabla\ell(u)\Vert^2)}
<&
 \inf_{v\in S(N_u)}\inf_{w\in M\setminus\{u\}:\langle w,v\rangle > 0}\frac{\sigma(u)/\sigma(w) - \langle w,\psi(u)\rangle}{\langle w,v\rangle} \\
=&
 \inf_{w\in M\setminus\{u\}}\inf_{v\in S(N_u):\langle w,v\rangle > 0}\frac{\sigma(u)/\sigma(w) - \langle w,\psi(u)\rangle}{\langle w,v\rangle} \\
=&
 \inf_{w\in M\setminus\{u\}}\frac{\sigma(u)/\sigma(w) - \langle w,\psi(u)\rangle}{\Vert P^\perp_u w\Vert},
\end{align*}
where $P^\perp_u$ is defined in the statement of Theorem \ref{thm:critical}.
Here we used
\[
 \sup_{v\in S(N_u):\langle w,v\rangle>0}\langle w,v\rangle = \Vert P^\perp_u w\Vert.
\]

(ii) When $M_u^-\ne\emptyset$,
\begin{align}
\label{Bu2}
B_u=&\,
\biggl\{(r,s,v)\in Q_u \mid
 \mbox{``}\langle w,v\rangle < 0\ \mbox{and}\ %
 \frac{s}{r} > \frac{- \sigma(u)/\sigma(w) + \langle w,\psi(u)\rangle}
            {-\langle w,v\rangle}\mbox{''},\ \ \forall w\in M\setminus\{u\}
\biggr\} \\
=&\,
\biggl\{(r,s,v)\in Q_u \mid
 \sup_{w\in M\setminus\{u\}}\langle w,v\rangle < 0,\ %
 \frac{s}{r} > \sup_{w\in M\setminus\{u\}}\frac{- \sigma(u)/\sigma(w) + \langle w,\psi(u)\rangle}
            {-\langle w,v\rangle}
\biggr\}. \nonumber
\end{align}
If $A_u(b)$ is not empty, then there exist $(r,s,v),(r,s,-v)\in A_u(b)$, but this contradicts the definition of $B_u$.
Therefore, $A_u(b)\subset B_u$ holds iff $A_u(b)=\emptyset$.
That is,
\[
 \frac{\sigma(u)^2}{b^2}-(1+\Vert\nabla\ell(u)\Vert^2) \le 0.
\]

Combining (i) and (ii), $\bc'(u)$ is given by
\[
 \sqrt{\frac{\sigma(u)^2}{\bc'(u)^2}-(1+\Vert\nabla\ell(u)\Vert^2)} = \inf_{w\in M\setminus\{u\}}\frac{\{\sigma(u)/\sigma(w) - \langle w,\psi(u)\rangle\}_+}{\Vert P^\perp_u w\Vert},
\]
or equivalently,
\[
 \bc'(u)^2 = \frac{\sigma(u)^2}{1+\inf_{w\in M\setminus\{u\}}h(u,w)}, 
\]
where $h(u,w)$ is defined in (\ref{h}).
The threshold $\bc=\sup_{u\in M}\bc'(u)$ is derived as (\ref{critical1}).
\end{proof}

\begin{proof}[Proof of Proposition \ref{prop:supporting}]
(i) is just a restatement.

(ii).
Let $\hu(u)=\inf_{w:w\ne u}h(u,w)$. Suppose that $u\notin\Msupp$.
Then, $h(u,w)=\Vert\nabla\ell(u)\Vert^2$ for some $\exists w\,(\ne u$),
 $\hu(u)=\Vert\nabla\ell(u)\Vert^2$,
\[
 \frac{1+\Vert\nabla\ell(u)\Vert^2}{\sigma(u)^2}\bc^2
 = \frac{1+\hu(u)}{\sigma(u)^2}\bc^2
 = \frac{1+\hu(u)}{\sigma(u)^2}\sup_{u'} \frac{\sigma(u')^2}{1+\hu(u')}\ge 1,
\]
which means that $u\notin\Mc$.
Therefore, $\Msupp\supset\Mc$.

(iii).
We follow the proof of Theorem \ref{thm:critical}.
Recall that $u\notin\Msupp \iff M_u^-\ne\emptyset$.
When $M_u^-\ne\emptyset$,
\[
\{(r,s,v)\in Q_u \mid \sigma(u)\langle u,z/\Vert z\Vert\rangle > b,\,z=r\psi(u)+sv\in V_u\} = A_u(b)\cap B_u = \emptyset
\]
for $b\ge\bc'(u)$, meaning that there is no point $z/\Vert z\Vert$ such that the maximum of $\sigma(u)\langle z/\Vert z\Vert,u\rangle$ is attained at $u$.
In other words, the radius of tube $(M)_{b/\sigma(\cdot)}$ vanishes at $u$ for any $b\ge\bc'(u)$.
\end{proof}

\begin{lemma}
\label{lem:nnd}
For $z\in V_u$, the matrix $G_u+C_u+(s/r)H_u(v)$ is non-negative definite.
\end{lemma}

\begin{proof}
From (\ref{Vu}),
\begin{equation}
\label{liminf}
 \liminf_{w\to u} \frac{ \sigma(u)/\sigma(w)
   - \langle w,\psi(u)\rangle
 - (s/r) \langle w,v\rangle }{\Vert w-u\Vert^2} \ge 0.
\end{equation}
To evaluate the left-hand side, we introduce a local coordinate system
$t=(t^i)$ so that $u=u(0)$, $w=u(t)$.
Because
\begin{align*}
 w =&\; u + \partial_i u t^i 
 + \frac{1}{2}\partial_i \partial_j u t^i t^j + o(\Vert t\Vert^2), \\
 \ell(w) =&\; \ell(u) + \partial_i \ell(u) t^i 
 + \frac{1}{2}\partial_i \partial_j \ell(u) t^i t^j + o(\Vert t\Vert^2),
\end{align*}
we obtain
\[
 \Vert w-u\Vert^2 = 2(1-\langle w,u\rangle) = g_{ij}(u)t^i t^j
 + o(\Vert t\Vert^2), \quad
 \langle w,v\rangle = - \frac{1}{2} h_{ij}(u;v) t^i t^j
 + o(\Vert t\Vert^2),
\]
where $g_{ij}(u)$ and $h_{ij}(u;v)$ are the $(i,j)$th elements of $G_u$ and $H_u(v)$, respectively, and
\begin{align*}
\frac{\sigma(u)}{\sigma(w)} - \langle w,\psi(u)\rangle
=& e^{-\ell(w)+\ell(u)} - \langle w,u-\nabla\ell(u)\rangle \\
=& 1 - \partial_i\ell(u) t^i
 - \frac{1}{2} \partial_i \partial_j \ell(u) t^i t^j
 + \frac{1}{2} \partial_i \ell(u) \partial_j \ell(u) t^i t^j \\
& -\langle u + \partial_i u t^i + \frac{1}{2}\partial_i \partial_j u t^i t^j,
 u - g^{kl}(u) \partial_k\ell(u) \partial_l u \rangle + o(\Vert t\Vert^2) \\
=& 1 - \partial_i\ell(u) t^i
 - \frac{1}{2} \partial_i \partial_j \ell(u) t^i t^j
 + \frac{1}{2} \partial_i \ell(u) \partial_j \ell(u) t^i t^j \\
& - 1 + \partial_i\ell(u) t^i
 - \frac{1}{2} \langle \partial_i \partial_j u,u\rangle t^i t^j
 + \frac{1}{2}\langle \partial_i \partial_j u,\partial_l u \rangle
 g^{kl}(u) \partial_k\ell(u) t^i t^j \\
& + o(\Vert t\Vert^2) \\
=& \frac{1}{2} \bigl( g_{ij}(u) - \nabla_i\nabla_j\ell(u)
 + \nabla_i\ell(u)\nabla_j\ell(u) \bigr) t^i t^j + o(\Vert t\Vert^2).
\end{align*}
(\ref{liminf}), by letting $\Vert t\Vert\to 0$, means that
the matrix $(G_u + C_u + (s/r) H_u(v)) G_u^{-1}$
with $C_u=(c_{ij}(u))$,
$c_{ij}(u) = - \nabla_i\nabla_j\ell(u) + \nabla_i\ell(u)\nabla_j\ell(u)$,
is non-negative definite.
\end{proof}

\begin{proof}[Proof of Proposition \ref{prop:validity}]
We prove the existence of a constant $c<1$ such that
\begin{equation}
\label{c}
 \frac{\sigma(u)^2}{1+h(u,w)} \le c\,\sigma_0^2 \quad \forall (u,w)\in M\times M,\,u\ne w.
\end{equation}
We divide $M\times M\setminus\{(u,u)\mid u\in M\}$ into three parts.

(i)
Recall that $M_0=\{u\in M \mid \sigma(u)=\sigma_0\}$ is the maximizer of $\sigma(u)$.
Suppose that $u\in M$ is in the neighborhood of $M_0$, and $w$ is sufficiently close to $u$.
By introducing the same local coordinates $u=u(0)$, $w=u(t)$ as in the proof of Lemma \ref{lem:nnd}, we see
\begin{equation}
\label{positive}
\frac{\sigma(u)}{\sigma(w)} -1 + \langle w,\nabla\ell(u)\rangle
= \frac{1}{2} \bigl( g_{ij}(u) - \nabla_i\nabla_j\ell(u)
 + \nabla_i\ell(u)\nabla_j\ell(u) \bigr) t^i t^j + o(\Vert t\Vert^2),
\end{equation}
which is positive because at $u\in M_0$, the gradient is $\nabla_i\ell(u)=0$, and the Hesse matrix $(\nabla_i\nabla_j\ell(u))$ is non-positive definite.
Let $M_1\supset M_0$ be an open subset of $M$ and $\varepsilon >0$ such that (\ref{positive}) is positive for all $u\in M_1$ and $\langle w,u\rangle>1-\varepsilon$.

Because
\[
 \frac{\sigma(u)}{\sigma(w)} - \langle w,\psi(u)\rangle
- (1 -  \langle w,u\rangle)
= \frac{\sigma(u)}{\sigma(w)} - 1 + \langle w,\nabla\ell(u)\rangle > 0,
\]
we have
\[
 h(u,w) \ge
 \frac{\{\sigma(u)/\sigma(w)
 - \langle w,\psi(u)\rangle\}_+}{\Vert P_u^\perp w\Vert}
 > \frac{1-\langle w,u\rangle}{\Vert P_{u}^\perp w\Vert}.
\]
Moreover,
\[
 \frac{1-\langle w,u\rangle}{\Vert P_{u}^\perp w\Vert}
\ge
 \inf_{u,w\in M,u\ne w} \frac{1-\langle w,u\rangle}{\Vert P_{u}^\perp w\Vert}
 = \tan^2\tc,
\]
where $\tc$ is the global critical radius of $M$ when $\sigma(u)\equiv 1$,
and is positive (Remark \ref{rem:critical}).

In summary, on the region
\begin{equation}
\label{set1}
 \{ (u,w)\in M_1\times M \mid \langle w,u\rangle>1-\varepsilon,\,u\ne w\},
\end{equation}
\[
 \frac{\sigma(u)^2}{1+h(u,w)}\le \sigma_0^2\cos^2\tc.
\]

(ii)
Let $\varepsilon>0$ be the number chosen in (i).
We can choose an open subset $M_1'$ of $M$, $M_0\subset M_1'\subset M_1$ such that $\sigma(u)> \sigma_0(1-\varepsilon^2)$ and $\Vert\nabla\ell(u)\Vert<\varepsilon^2$ for $u\in M_1'$.
Then, for $u\in M_1'$ and $w$ such that $\langle w,u\rangle\le 1-\varepsilon$,
\[
 h(u,w)\ge
 \frac{\sigma(u)}{\sigma(w)} - \langle w,u-\nabla\ell(u)\rangle
 \ge \frac{\sigma(u)}{\sigma(w)} - \langle w,u\rangle - \Vert\nabla\ell(u)\Vert \ge \varepsilon-2\varepsilon^2=\varepsilon',
\]
which means that on the region
\begin{equation}
\label{set2}
 \{(u,w)\in M_1'\times M\mid \langle w,u\rangle\le 1-\varepsilon\},
\end{equation}
\[
 \frac{\sigma(u)^2}{1+h(u,w)} \le \sigma_0^2/(1+\varepsilon').
\] 

(iii)
$M_1'$ contains $M_0$, and outside $M_1'$ the $\sigma(u)$ is strictly less that $\sigma_0=\max_{u\in M}\sigma(u)$.
Let $\sigma_1=\max_{u\in M\setminus M_1'}\sigma(u)<\sigma_0$.
On the region
\begin{equation}
\label{set3}
 \{(u,w)\in (M\setminus M_1')\times M \mid u\ne w\},
\end{equation}
$h(u,w)\ge 0$ and
\[
 \frac{\sigma(u)^2}{1+h(u,w)} \le \sigma_1^2.
\] 

The union of (\ref{set1}), (\ref{set2}) and (\ref{set3}) covers $M\times M\setminus\{(u,u)\mid u\in M\}$.
Therefore, we have (\ref{c}) with the constant
\[
 c=\max\bigr\{\cos^2\tc,1/(1+\varepsilon'),\sigma_1^2/\sigma_0^2\bigl\} < 1.
\]
\end{proof}

Lemma \ref{lem:nnd} affirms that $p(u,r,s,v)$ in (\ref{p}) is non-negative.
We have the refinement of (\ref{PU2}) as follows.

\begin{proposition}
Let $\dd u$ and $\dd v$ be the volume elements of $M$ and $S(N_u)$, respectively.
Let $\bc$ be the threshold defined in (\ref{critical}).
Then, for all $b$,
\[
\P\bigl(\Ymax > b\bigr)
= \int_M \dd u \int_0^\infty \dd r \int_0^\infty \dd s
\,\1_{\Bigl\{\frac{\sigma(u)r}{\sqrt{r^2(1+\Vert\nabla\ell\Vert^2)+s^2}} > b\Bigr\}\cap B_u}
 \int_{S(N_u)} \dd v \, p(u,r,s,v),
\]
where
$B_u=(\ref{Bu1})$ if $M_u^-=\emptyset$,
$(\ref{Bu2})$ if $M_u^-\ne\emptyset$.
For $b\ge\bc$,
\begin{equation}
\label{PU3}
\P\bigl(\Ymax > b\bigr)
= \int_M \dd u \int_0^\infty \dd r \int_0^\infty \dd s
\,\1_{\Bigl\{\frac{\sigma(u)r}{\sqrt{r^2(1+\Vert\nabla\ell\Vert^2)+s^2}} > b\Bigr\}}
 \int_{S(N_u)} \dd v \, p(u,r,s,v).
\end{equation}
\end{proposition}

\subsection{Volume of tube formula: Intrinsic representation}

In this subsection, we prove Theorem \ref{thm:volume}.

We evaluate the integral (\ref{PU3}).
The following lemma is useful for managing the determinant factor in $p(u,r,s,v)$ of (\ref{p}).

For a finite set $I=\{i_1,\ldots,i_n\}$, the notations $\Sigma(I)$ and $\Pi(I)$ are defined in (\ref{SI}) and (\ref{PI}), respectively.
Recall that $\Sigma_d=\Sigma(\{1,\ldots,d\})$ and $\Pi_d=\Pi(\{1,\ldots,d\})$.

\begin{lemma}
\label{lem:wick}
Let $H=(h_{ij})$ be a $d\times d$ square Gaussian random matrix of mean zero.
Let
\[
 s_{ij;kl} = \E[h_{ik} h_{jl} - h_{il} h_{jk}] = \E\left[\det
 \begin{pmatrix} h_{ik} & h_{il} \\ h_{jk} & h_{jl}\end{pmatrix}\right].
\]
Then,
\begin{equation}
\label{EdetH}
 \E[\det(H)] =
 \begin{cases}
 \displaystyle
 \frac{1}{(d/2)!}\sum_{\pi,\tau\in \Pi_d} \sgn(\pi)\sgn(\tau)
 s_{\pi_1\pi_2;\tau_1\tau_2}s_{\pi_3\pi_4;\tau_3\tau_4}
 \cdots s_{\pi_{d-1}\pi_{d};\tau_{d-1}\tau_{d}} & (\mbox{if $d$ is even}), \\
 0 & (\mbox{if $d$ is odd}).
\end{cases}
\end{equation}
\end{lemma}

\begin{proof}
The case $d$ is odd is easy.
Suppose that $d$ is even.
\begin{align*}
 \E[\det(H)]
=& \sum_{\tau\in \Sigma_{d}} \sgn(\tau) \E[h_{1\tau_1}h_{2\tau_2}\cdots h_{d,\tau_{d}}] \\
=& \sum_{\tau\in \Sigma_{d}} \sgn(\tau) \frac{1}{(d/2)!}\sum_{\pi\in \Pi_{d}}\E\bigl[h_{\pi_1\tau_{\pi_1}}h_{\pi_2\tau_{\pi_2}}\bigr]\E\bigl[h_{\pi_3\tau_{\pi_3}}h_{\pi_4\tau_{\pi_4}}\bigr]\cdots\E\bigl[h_{\pi_{d-1},\tau_{\pi_{d-1}}}h_{\pi_{d},\tau_{\pi_{d}}}\bigr] \\
=& \frac{1}{(d/2)!} \sum_{\tau\in \Sigma_{d}} \sum_{\pi\in \Pi_{d}}\sgn(\tau) \sgn(\pi)
 \E[h_{\pi_1\tau_1}h_{\pi_2\tau_2}]
 \cdots \E[h_{\pi_{d-1},\tau_{d-1}}h_{\pi_{d},\tau_{d}}] \\
=& \frac{1}{(d/2)!} \sum_{\tau\in \Pi_{d}} \sum_{\pi\in \Pi_{d}}\sgn(\tau) \sgn(\pi)
\Biggl(\sum_{\sigma\in \Sigma(\{1,2\})}\sgn(\sigma) \E\bigl[h_{\pi_1\tau_{\sigma_1}}h_{\pi_2\tau_{\sigma_2}}\bigr]\Biggr)
 \times\cdots \\
& \qquad
 \times\Biggl(\sum_{\sigma\in \Sigma(\{d-1,d\})}\sgn(\sigma)
 \E\bigl[h_{\pi_{d-1},\tau_{\sigma_1}}h_{\pi_{d},\tau_{\sigma_2}}\bigr]\Biggr) \\
=& \frac{1}{(d/2)!} \sum_{\tau\in \Pi_{d}} \sum_{\pi\in \Pi_{d}}\sgn(\tau) \sgn(\pi)
 \bigl(\E[h_{\pi_1\tau_1}h_{\pi_2\tau_2}]-\E[h_{\pi_1\tau_2}h_{\pi_2\tau_1}]\bigr) \times\cdots \\
& \qquad
 \times\bigl(\E[h_{\pi_{d-1},\tau_{d-1}}h_{\pi_{d},\tau_{d}}]-\E[h_{\pi_{d-1},\tau_{d}}h_{\pi_{d},\tau_{d-1}}]\bigr),
\end{align*}
which is the right-hand side of (\ref{EdetH}).
The second equality is due to the moment formula for a product of Gaussian random variables.
\end{proof}

\begin{remark}
Lemma \ref{lem:wick} can be stated in terms of the commutative algebra of double forms.
See \cite{federer:1959},\cite{taylor-adler:2003},\cite{kuriki-takemura:2009}.
\end{remark}

\begin{proof}[Proof of Theorem \ref{thm:volume}]
We apply Lemma \ref{lem:wick} in the evaluation of the integral (\ref{PU3}) with respect to $\dd v$.
Recall that $\dd v$ is the volume element of the ($n-d-2$)-dimensional unit sphere $S(N_u)$,
and hence, $\dd v/\Omega_{n-d-1}$ is the uniform distribution on the sphere.

Suppose that $\eta$ is the standard Gaussian distribution of space $N_u=T_u(M)^\perp\cap T_u(\S^{n-1})$.
Then, $\eta/\Vert \eta\Vert\sim \dd v/\Omega_{n-d-1}$ and $\Vert\eta\Vert^2\sim\chi^2_{n-d-1}$ are independently distributed.
Let $\HH(v)=(\h_{ij}(v))=H(v)(G+C)^{-1}$ for short.
The determinant in $p(u,r,s,v)$ in (\ref{p}) is expanded as
\begin{align}
\label{detI+H}
\det\bigl(I_d+(s/r)\HH(v)\bigr)
 =& \sum_{e=0}^d (s/r)^e \tr_e\bigl(\HH(v)\bigr) \\
 =& \sum_{e=0}^d (s/r)^e \!\!\!\!\sum_{I:\,I\subset\{1,\ldots,d\},\,|I|=e} \det\bigl(\HH_I(v)\bigr), \qquad 
\HH_I(v)=\bigl(\h_{ij}(v)\bigr)_{i,j\in I}.
\nonumber
\end{align}
Hence, the integral with respect to $\dd v$ involved in (\ref{PU3}) is
\begin{align}
\label{EdetHI}
 \int \det\bigl(\HH_I(v)\bigr) \dd v 
=& \Omega_{n-d-1}\E\bigl[\det\HH_I(\eta/\Vert\eta\Vert)\bigr] \\
=& \frac{\Omega_{n-d-1}}{\E\bigl[\Vert\eta\Vert^e\bigr]} \E\bigl[\det\HH_I(\eta)\bigr]
= \frac{\Omega_{n-d-1}}{\E\bigl[(\chi^2_{n-d-1})^{e/2}\bigr]} \E\bigl[\det\HH_I(\eta)\bigr] \nonumber \\
=& \frac{\Omega_{n-d+e-1}}{(2\pi)^{e/2}} \E\bigl[\det\HH_I(\eta)\bigr],
\nonumber
\end{align}
and $\E\bigl[\det\HH_I(\eta)\bigr]$ is expressed by (\ref{EdetH}) with
$s_{ij,kl} = \RR_{ij}^{kl}$ in (\ref{tildeR}),
because the curvature tensor is characterized as
\begin{equation}
\label{curvature}
 R_{ij;kl} = \bigl(g_{ik}g_{jl}-g_{il}g_{jk}\bigr)
 + \E[h_{ik}(\eta)h_{jl}(\eta)-h_{il}(\eta)h_{jk}(\eta)],
\end{equation}
where $H(v)=(h_{ij}(v))$ is the second fundamental form,
and $\eta$ is a standard Gaussian vector in $N_u$
(\cite{weyl:1939}).
(\ref{curvature}) is the Gauss equation for a submanifold of the unit sphere.

On the other hand, the integral with respect to $\dd r$ and $\dd s$ in (\ref{PU3}) is a linear combination of
\begin{align}
\label{int_rs}
& \frac{1}{(2\pi)^{n/2}} \int_0^\infty \int_0^\infty \1_{\bigl\{\frac{r^2(1+\Vert\nabla\ell\Vert^2)}{r^2(1+\Vert\nabla\ell\Vert^2)+s^2}>\frac{1+\Vert\nabla\ell\Vert^2}{\sigma^2}b^2\bigr\}} r^d s^{n-d-2} (s/r)^e e^{-\frac{1}{2}\{r^2(1+\Vert\nabla\ell(u)\Vert^2) + s^2\}} \dd r\,\dd s \\
&= (1+\Vert\nabla\ell\Vert^2)^{-\frac{1}{2}(d-e+1)} \frac{1}{(2\pi)^{n/2}} \int_0^\infty \int_0^\infty \1_{\bigl\{\frac{r^2}{r^2+s^2}>\frac{1+\Vert\nabla\ell\Vert^2}{\sigma^2}b^2\bigr\}} r^{d-e} s^{n-d+e-2} e^{-\frac{1}{2}(r^2 + s^2)} \dd r\,\dd s \nonumber \\
&= (1+\Vert\nabla\ell\Vert^2)^{-\frac{1}{2}(d-e+1)} \frac{1}{2(2\pi)^{n/2}} \int_{\frac{1+\Vert\nabla\ell\Vert^2}{\sigma^2}b^2}^1 t^{(d-e-1)/2} (1-t)^{(n-d+e-3)/2} \dd t \int_0^\infty e^{-\frac{1}{2}R^2} R^{n-1} \dd R \nonumber \\
&= (1+\Vert\nabla\ell\Vert^2)^{-\frac{1}{2}(d-e+1)}
 \frac{1}{\Omega_{d-e+1}\Omega_{n-d+e-1}}
 \bar B_{\frac{1}{2}(d-e+1),\frac{1}{2}(n-d+e-1)}
   \biggl(\frac{1+\Vert\nabla\ell\Vert^2}{\sigma^2} b^2\biggr).
\nonumber
\end{align}
Combining (\ref{detI+H}), (\ref{EdetHI}), and (\ref{int_rs}),
we have (\ref{PU1}).

Moreover, the point $u$ such that the argument of $\bar B_{\frac{1}{2}(d-e+1),\frac{1}{2}(n-d+e-1)}(\cdot)$ is always greater than or equal to 1 makes no contribution to (\ref{PU1}).
Hence, the range of integration $M$ in (\ref{PU1}) can be replaced with $\Mc$.
The proof of Theorem \ref{thm:volume} is completed.
\end{proof}

\subsection{Tail probability of the maximum}
\label{subsec:proof-tail}

In this subsection, we prove Theorem \ref{thm:tube_method}.

\begin{proof}[Proof of Theorem \ref{thm:tube_method}]
The right-hand side of (\ref{PU1}) is denoted by $F(b)$.
Because $\P\bigl(\Ymax >~\cdot\bigr)$ is monotonically decreasing,
Theorem \ref{thm:volume} implies that
\[
 \1(b\ge \bc) F(b) \le \P\bigl(\Ymax > b\bigr) \le \1(b\ge \bc) F(b) + \1(b<\bc),
\]
or
\[
 -\1(b<\bc) F(b) \le \P\bigl(\Ymax > b\bigr)-F(b) \le -\1(b<\bc) F(b) + \1(b<\bc).
\]
We use (\ref{PT}).
By letting $b:=c/\Vert\xi\Vert^2$,
and taking the expectation with respect to $\Vert\xi\Vert^2\sim\chi^2_n$, 
\[
 |\P\bigl(\Xmax > c\bigr)-\E[F(c/\Vert\xi\Vert^2)]| \le (C+1)\P(c/\Vert\xi\Vert^2 < \bc)
 = (C+1)\bar G_n(c/\bc),
\]
where $C=\max_{0\le b\le \sigma_0} |F(b)|<\infty$, $\sigma_0=\max_{u\in M}\sigma(u)$.
Noting that the relationship
\[
 \E\Bigl[\bar B_{\frac{1}{2}k,\frac{1}{2}(n-k)}(c^2/\Vert\xi\Vert^2)\Bigr] = \bar G_k(c^2), \quad \Vert\xi\Vert^2\sim\chi^2_n,
\]
we see that $\E[F(c/\Vert\xi\Vert^2)]$ is $\Ptube\bigl(\Xmax > c\bigr)$ in (\ref{PT2}).
\end{proof}

\subsection{Laplace approximation}

In this subsection, we prove Theorem \ref{thm:laplace} and Corollary \ref{cor:laplace}.

\begin{proof}[Proof of Theorem \ref{thm:laplace} and Corollary \ref{cor:laplace}]
The leading term of the tube formula (\ref{PT2}) is
\begin{align*}
\Ptube\bigl(\Xmax > c\bigr)
\sim& \frac{1}{\Omega_{d+1}}
\int_M \dd u \frac{\det(I_d+CG^{-1})}{(1+\Vert\nabla\ell\Vert^2)^{\frac{1}{2}(d+1)}} \, \bar G_{d+1}\biggl( \frac{1+\Vert\nabla\ell\Vert^2}{\sigma^2} c^2 \biggr) \\
=& \frac{1}{\Omega_{d+1}}
\int_M \dd u \frac{\det(I_d+CG^{-1})}{(1+\Vert\nabla\ell\Vert^2)^{\frac{1}{2}(d+1)}} \int_{t>\frac{1+\Vert\ell\Vert^2}{\sigma^2}c^2} \dd t\,t^{\frac{1}{2}(d+1)-1} e^{-t/2} \\
=& \frac{1}{(2\pi)^{\frac{1}{2}(d+1)}}
 \int_M \dd u \det(I_d+CG^{-1})\int_{r>c} \frac{\dd r \,r^d}{\sigma^{d+1}}
 e^{-\frac{1}{2}\bigl(\frac{r}{\sigma}\bigr)^2 (1+\Vert\nabla\ell\Vert^2) }.
\end{align*}

Here, we treat the two typical cases:
(i) The maximum of the variance is attained at only $M_0=\{u_0\}\subset M$.
(ii) The minimizer of variance $M_0$ forms a $d_0$-dimensional closed manifold.

The values evaluated at $u=u_0$ are denoted by the subscript 0, i.e.,
$\sigma_0=\sigma(u_0)$, $G_0=G_{u_0}$, and $C_0=C_{u_0}$.

Case (i).
Let
\[
 q = \frac{1}{\sigma^2}(1+\Vert\nabla\ell\Vert^2) = e^{-2\ell}(1+\ell_k\ell_l g^{kl}).
\]
The derivative with respect to $t^i$ is denoted by the subscript, for example, $\ell_i=\partial\ell/\partial t^i$.
Because $\ell_i|_{u_0}=0$,
\begin{align*}
\partial_i q
=& -2 \ell_i e^{-2\ell}(1+\ell_k\ell_l g^{kl})
 + e^{-2\ell}(2\ell_{ik}\ell_l g^{kl}+\ell_k\ell_l\partial_i g^{kl}), \\
\partial_i q|_{u_0}
=& 0, \\
\partial_i\partial_j q|_{u_0}
=& -2 \ell_{ij} e^{-2\ell}
 + e^{-2\ell}(2\ell_{ik}\ell_{jl} g^{kl})
= \frac{2}{\sigma^2}(-\ell_{ij}+\ell_{ik}g^{kl}\ell_{lj}).
\end{align*}
In matrix notation,
\[
 \bigl(\partial_i\partial_j q|_{u_0}\bigr)_{1\le i,j\le d} = C_0 + C_0 G_0^{-1} C_0
\]
and about the point $u_0=u(t)|_{t=0}$,
\begin{equation}
\label{q}
 q = \frac{1}{\sigma_0^2} \bigl(1 + t^\top (C_0 + C_0 G_0^{-1} C_0) t\bigr) + o(\Vert t\Vert^2), \ \ t=(t^1,\ldots,t^d)^\top,
\end{equation}
where $u_0=\varphi(0)$.
The volume element at $u_0$ is $\dd u=\det(G_0)^{\frac{1}{2}}\prod_{i=1}^d \dd t_i$.
Therefore, according to the standard arguments of Laplace's method (e.g., \cite{erdelyi:1956}),
\begin{align*}
\Ptube\bigl(\Xmax > c\bigr)
\sim& \frac{1}{(2\pi)^{\frac{1}{2}(d+1)}}
 \det(I_d+C_0 G_0^{-1}) \int_{r>c} \frac{\dd r\,r^d}{\sigma_0^{d+1}} \,
 e^{-\frac{1}{2}\bigl(\frac{r}{\sigma_0}\bigr)^2 } \\
& \times (2\pi)^{d/2}
 \det(C_0+C_0G_0^{-1}C_0)^{-\frac{1}{2}}\det(G_0)^{\frac{1}{2}}
 \biggl(\frac{r}{\sigma_0}\biggr)^{-d} \\
\sim& \det(I_d+C_0 G_0^{-1})^{\frac{1}{2}} \det(C_0 G_0^{-1})^{-\frac{1}{2}}
  \bar\Phi\biggl(\frac{c}{\sigma_0}\biggr),
\end{align*}
where
\[
 C_0 = (-\nabla_i\nabla_j\ell(u_0))_{1\le i,j\le d}
 = (-\nabla_i\nabla_j\log\sigma(u_0))_{1\le i,j\le d}.
\] 

Case (ii).
When $\dim M_0=d_0>0$, the matrix $C+C G^{-1} C$ evaluated at $u\in M_0$ in (\ref{q}) is of rank $d-d_0$.
By redefining the local coordinates $t=(t^{i})$, we can assume that $C$ at $u$ is of the form $C = \diag(0, C_{22})$ where $C_{22}$ is $(d-d_0)\times(d-d_0)$.
Accordingly, we write
\[
  G = \begin{pmatrix} G_{11} & G_{12} \\ G_{21} & G_{22} \end{pmatrix}, \ \ \mbox{and let}\ \ G_{22\cdot 1}=G_{22}-G_{21}G_{11}^{-1}G_{12}.
\]
Then, for $u\approx u_0=u(t)|_{t=0}\in M_0$,
\[
 q(u) = \frac{1}{\sigma_0^2} \bigl(1 + \bar t^\top (C_{22} + C_{22} G_{22\cdot 1}^{-1} C_{22})|_{u_0} \bar t\bigr) + o(\Vert \bar t\Vert^2), \ \ \bar t=\bigl(t^{d_0+1},\ldots,t^d\bigr)^\top,
\]
\[
 \dd u = \dd u_0 \prod_{i=d_0+1}^{d} \dd t^i \det(G_{22\cdot 1})^{\frac{1}{2}}, \quad \dd u_0=\det(G_{11})^{\frac{1}{2}}\prod_{i=1}^{d_0} \dd t^i,
\]
and
\begin{align*}
\Ptube\bigl(\Xmax > c\bigr)
\sim& \frac{1}{(2\pi)^{\frac{1}{2}(d+1)}}
 \int_{M_0} \dd u_0 \det(I_{d-d_0}+C_{22}G_{22\cdot 1}^{-1})\int_{r>c} \frac{\dd r\,r^d}{\sigma_0^{d+1}} 
 e^{-\frac{1}{2}\bigl(\frac{r}{\sigma_0}\bigr)^2 } \\
& \times (2\pi)^{\frac{1}{2}(d-d_0)}
 \det(C_{22}+C_{22}G_{22\cdot 1}^{-1}C_{22})^{-\frac{1}{2}}\det(G_{22\cdot 1})^{\frac{1}{2}}
 \biggl(\frac{r}{\sigma_0}\biggr)^{-(d-d_0)} \\
\sim&
 \frac{1}{\Omega_{d_0+1}}
 \int_{M_0} \dd u_0 \frac{\det(I_{d-d_0}+C_{22}G_{22\cdot 1}^{-1})^{\frac{1}{2}}}{ \det(C_{22}G_{22\cdot 1}^{-1})^{\frac{1}{2}}} \times \bar G_{d_0+1}\biggl(\frac{c^2}{\sigma_0^2}\biggr) \\
=&
 \frac{1}{\Omega_{d_0+1}}
 \int_{M_0} \dd u_0 \frac{\det(I_d+C G^{-1})^{\frac{1}{2}}}{ \tr_{d-d_0}(C G^{-1})^{\frac{1}{2}}} \times \bar G_{d_0+1}\biggl(\frac{c^2}{\sigma_0^2}\biggr).
\end{align*}
The asymptotic equivalence on the last line is due to the uniformity of the Laplace approximation on the compact set $M_0$.
\end{proof}

\subsection{Approximation error of the Bonferroni method}

In this subsection, we prove Theorem \ref{thm:bonferroni} and Proposition \ref{prop:bonferroni-validity}.
For this purpose, we will provide analogues to Theorems \ref{thm:volume} and \ref{thm:critical} when $M$ is a finite set.

\begin{proof}[Proof of Theorem \ref{thm:bonferroni}]
Recall that
$X_i=\sigma_i\langle\varphi^{(i)},\xi\rangle$,
$Y_i=\sigma_i\langle\varphi^{(i)},\xi/\Vert\xi\Vert\rangle$,
 $\xi\sim\mathcal{N}_n(0,I_n)$,
and $\rho_{ij}=\langle\varphi^{(i)},\varphi^{(j)}\rangle$.
If $b$ is sufficiently large, say $b\ge\bc$, the events $\{Y_i>b\}$ are mutually exclusive, and
\[
 \P\bigl(\max\{Y_1,\ldots,Y_K\}>b\bigr) = \sum_{i=1}^K \P(Y_i>b),
\]
which is (\ref{PU4}).

To identify the threshold $\bc$, we will follow the proof in Section \ref{subsec:proof-critical}.
Let $i$ be fixed.
Let
\[
 N_i = T_{\varphi^{(i)}}(\S^{n-1}) 
\]
be the space in $\R^n$ orthogonal to $\varphi^{(i)}$.
Almost all $z\in\R^n$ is written as
\[
 z = r \varphi^{(i)} + s v, \ \ (r,s,v)\in Q_i,
\]
where
\[
 Q_i = \{(r,s,v) \mid r>0,\ s\ge 0,\ v\in S(N_i)\}.
\]
Let
\[
 V_i = \{ z \in \R^n \mid \sigma_j\langle \varphi^{(j)},z\rangle
 < \sigma_i\langle \varphi^{(i)},z\rangle,\ z = r\varphi^{(i)} + s v,
 \ \forall j\ne i \}.
\]
Instead of (\ref{rsv}), we will find the threshold $\bc'(i)$ of $b$ satisfying 
$A_i(b)\subset B_i$ a.s., where
\begin{align*}
A_i(b) =&\, \{(r,s,v)\in Q_i \mid \sigma_i\langle \varphi^{(i)},z/\Vert z\Vert\rangle > b,\,z=r\varphi^{(i)}+sv,\ \forall j\ne i \}, \\
B_i   =&\, \{(r,s,v)\in Q_i \mid z=r\varphi^{(i)}+sv\in V_i\}.
\end{align*}
The threshold $\bc$ in Theorem \ref{thm:bonferroni} is given by $\max_{i\in I}\bc'(i)$.

As in (\ref{iii}), the set $A_i(b)$ is
\[
A_i(b)=\bigl\{ (r,s,v)\in Q_i \mid  (s/r)^2 < \sigma_i^2/b^2-1 \bigr\}.
\]

Let $I=\{1,\ldots,K\}$.
We divide the index set $I\setminus\{i\}$ as $I_i^+\sqcup I_i^-$, where
\[
 I_i^+ = \bigl\{ j\in I \mid \sigma_i/\sigma_j
   - \rho_{ij} > 0 \bigr\}, \qquad
 I_i^- = \bigl\{ j\in I \mid \sigma_i/\sigma_j
   - \rho_{ij} \le 0 \bigr\} \setminus\{i\}.
\]

(i) When $I_i^-=\emptyset$,
\begin{align*}
B_i
=&
\biggl\{ (r,s,v)\in Q_i \mid 
 \frac{s}{r} <
 \min_{j\in I\setminus\{i\}:\langle \varphi^{(j)},v\rangle > 0}\frac{\sigma_i/\sigma_j - \rho_{ij}}{\langle \varphi^{(j)},v\rangle} \biggr\}.
\nonumber
\end{align*}
Therefore, $A_i(b)\subset B_i$ iff
\begin{align*}
 \sqrt{\frac{\sigma_i^2}{b^2}-1}
<&
 \inf_{v\in S(N_i)}\min_{j\in I\setminus\{i\}:\langle \varphi^{(j)},v\rangle > 0}\frac{\sigma_i/\sigma_j - \rho_{ij}}{\langle \varphi^{(j)},v\rangle} \\
=&
 \min_{j\in I\setminus\{i\}}\inf_{v\in S(N_i):\langle \varphi^{(j)},v\rangle > 0}\frac{\sigma_i/\sigma_j - \rho_{ij}}{\langle \varphi^{(j)},v\rangle}
=
 \min_{j:j\ne i}\frac{\sigma_i/\sigma_j - \rho_{ij}}{\sqrt{1-\rho_{ij}^2}}.
\end{align*}
Here we used
\[
 \sup_{v\in S(N_i):\langle \varphi^{(j)},v\rangle>0}\langle \varphi^{(j)},v\rangle = \sqrt{1-\rho_{ij}^2}.
\]

(ii) When $I_i^-\ne\emptyset$,
\[
B_i=
\biggl\{(r,s,v)\in Q_i \mid
 \mbox{``}\langle \varphi^{(j)},v\rangle < 0\ \mbox{and}\ %
 \frac{s}{r} > \frac{- \sigma_i/\sigma_j + \rho_{ij}}
            {-\langle \varphi^{(j)},v\rangle}\mbox{''},\ \ \forall j\in I\setminus\{i\}
\biggr\}.
\]
$A_i(b)\subset B_i$ holds iff $A_i(b)=\emptyset$.
That is,
\[
 \frac{\sigma_i^2}{b^2}-1 \le 0.
\]

Combining (i) and (ii), $\bc'(i)$ is given by
\[
 \sqrt{\frac{\sigma_i^2}{\bc'(i)^2}-1} = \min_{j\in I\setminus\{i\}}\frac{(\sigma_i/\sigma_j - \rho_{ij})_+}{\sqrt{1-\rho_{ij}^2}},
\]
or equivalently,
\[
 \bc'(i)^2 = \frac{\sigma_i^2}{1+\min_{j\in I\setminus\{i\}}h(i,j)}, 
\]
where $h(i,j)$ is given in (\ref{hij}).
The threshold in Theorem \ref{thm:bonferroni} is $\bc=\max_{i\in i}\bc'(i)$.
In (\ref{PU4}), the term $\P(Y_i>b) = (1/2)\bar B_{\frac{1}{2},\frac{1}{2}(n-1)}(b^2/\sigma_i^2)$ for $\sigma_i<\bc$ is always zero and can be omitted.

Finally, we apply the proof of Theorem \ref{thm:tube_method} in Section \ref{subsec:proof-tail} to (\ref{PU4}) to prove (\ref{bonferroni}) of Theorem \ref{thm:bonferroni}.

\end{proof}

\begin{proof}[Proof of Proposition \ref{prop:bonferroni-validity}]
We evaluate
\begin{equation}
\label{(1+h)/s}
 \frac{1+h(i,j)}{\sigma_i^2} =
 \frac{1}{\sigma_i^2} \biggl(1 + \frac{(\sigma_i/\sigma_j - \rho_{ij})_+^2}{1 - \rho_{ij}^2} \biggr).
\end{equation}

(i) In the case $\rho_{ij}<1\le\sigma_i/\sigma_j$,
\[
 (\ref{(1+h)/s}) \ge \frac{1}{\sigma_i^2}\biggl(1 + \frac{(1 - \rho_{ij})^2}{1 - \rho_{ij}^2} \biggr) = \frac{1}{\sigma_i^2} \frac{2}{1+\rho_{ij}}.
\]

(ii) In the case $\rho_{ij}<\sigma_i/\sigma_j<1$,
\begin{align*}
 (\ref{(1+h)/s})
= \frac{1}{\sigma_i^2} \biggl(1 + \frac{(\sigma_i/\sigma_j-\rho_{ij})^2}{1-\rho_{ij}^2} \biggr)
=& \frac{1}{\sigma_j^2} \biggl(1 + \frac{(\sigma_j/\sigma_i-\rho_{ij})^2}{1-\rho_{ij}^2} \biggr) \\
>& \frac{1}{\sigma_j^2} \biggl(1 + \frac{(1-\rho_{ij})^2}{1-\rho_{ij}^2} \biggr)
= \frac{1}{\sigma_j^2} \frac{2}{1+\rho_{ij}}.
\end{align*}

(iii) In the case $\sigma_i/\sigma_j\le\rho_{ij}<1$,
\[
 (\ref{(1+h)/s}) = \frac{1}{\sigma_i^2} \ge \frac{1}{\sigma_j^2\rho_{ij}^2} > \frac{1}{\sigma_j^2} \frac{2}{1+\rho_{ij}}.
\]
In each case,
\[
 (\ref{(1+h)/s}) \ge \min\biggl\{\frac{1}{\sigma_i^2},\frac{1}{\sigma_j^2}\biggr\} \frac{2}{1+\rho_{ij}},
\]
and hence
\[
 \frac{1}{\bc^2} \ge \min_{i\ne j} \frac{1}{\sigma_i^2} \frac{2}{1+\rho_{ij}} \ge \biggl(\min_i \frac{1}{\sigma_i^2}\biggr)\biggl(\min_{i<j}\frac{2}{1+\rho_{ij}}\biggr).
\]
\end{proof}

\subsection*{Acknowledgments}

The authors are grateful to Robert Adler for bringing \cite{rosmarin:2015} to their attention.

\bibliographystyle{amsalpha}
\bibliography{inhomogeneous-bib.bib}

\end{document}